\documentclass[11pt]{article}

\usepackage[utf8]{inputenc}
\usepackage[T1]{fontenc}
\usepackage[USenglish]{babel}
\usepackage[a4paper,margin=32mm,headheight=14pt]{geometry}
\usepackage{lmodern}
\usepackage{microtype}
\usepackage{csquotes}
\usepackage{amsmath,amssymb}
\usepackage{amsthm}
\usepackage{xcolor}
\usepackage{tikz}
\usepackage{fancyhdr}
\usepackage[
  colorlinks=true,
  linkcolor=blue!45!black,
  citecolor=green!35!black,
  urlcolor=blue!55!black,
  pdfauthor={Alex Junior Gomez Saltachin},
  pdftitle={Categorical genus for log del Pezzo surfaces with cyclic quotient singularities},
  pdfsubject={Algebraic geometry and mirror symmetry},
  pdfkeywords={Log del Pezzo surfaces, canonical stacks, cyclic quotient singularities, categorical mirror genus, maximally mutable Laurent polynomials}
]{hyperref}
\usepackage[capitalize,nameinlink]{cleveref}

\theoremstyle{plain}
\newtheorem{theorem}{Theorem}[section]
\newtheorem{proposition}[theorem]{Proposition}
\newtheorem{lemma}[theorem]{Lemma}
\newtheorem{corollary}[theorem]{Corollary}
\newtheorem*{citedtheorem}{Theorem}
\theoremstyle{definition}
\newtheorem{definition}[theorem]{Definition}
\newtheorem{example}[theorem]{Example}
\theoremstyle{remark}
\newtheorem{remark}[theorem]{Remark}

\crefname{theorem}{Theorem}{Theorems}
\crefname{proposition}{Proposition}{Propositions}
\crefname{lemma}{Lemma}{Lemmas}
\crefname{corollary}{Corollary}{Corollaries}
\crefname{definition}{Definition}{Definitions}
\crefname{example}{Example}{Examples}
\crefname{remark}{Remark}{Remarks}

\newcommand{\cX}{\mathcal X}
\newcommand{\Db}{\mathrm D^b}
\newcommand{\coh}{\operatorname{coh}}
\newcommand{\NS}{\operatorname{NS}}
\newcommand{\rk}{\operatorname{rk}}
\newcommand{\Span}{\operatorname{Span}}
\newcommand{\Ext}{\operatorname{Ext}}
\newcommand{\C}{\mathbb C}
\newcommand{\Q}{\mathbb Q}
\newcommand{\Z}{\mathbb Z}
\newcommand{\Pj}{\mathbb P}
\newcommand{\1}{\mathbf 1}

\usepackage[
  backend=biber,
  style=numeric-comp,
  sorting=nyt,
  sortcites=true,
  giveninits=true,
  doi=true,
  url=false,
  isbn=false
]{biblatex}
\AtBeginBibliography{\small\setlength{\bibitemsep}{0.25\baselineskip}}

\fancypagestyle{plain}{%
  \fancyhf{}
  \fancyfoot[C]{\thepage}
  
  }
\begin{document}

\title{\bfseries Categorical genus for log del Pezzo surfaces with cyclic quotient singularities}
\author{Alex Junior Gomez Saltachin\\[0.5em]
\small Departamento de Matem\'atica, Pontif\'icia Universidade Cat\'olica do Rio de Janeiro\\
\small Rua Marqu\^es de S\~ao Vicente, 225, Rio de Janeiro, RJ 22451-900, Brazil\\
\small \href{mailto:alex.gomez@pucp.edu.pe}{alex.gomez@pucp.edu.pe}}
\date{}

\maketitle

\begin{abstract}
Let \(X\) be a log del Pezzo surface with cyclic quotient singularities and \(\mathcal X\) its canonical stack. We define the categorical genus of \(\mathcal X\) as half the rank of its antisymmetric numerical Euler pairing. Using Serre duality and inertial Riemann–Roch, we compute
\[
g_{\mathrm{cat}}(\mathcal X)=1+\frac12\sum_j w_j(\ell_j-1),
\]
where \(w_j\) and \(\ell_j\) are the local widths and Gorenstein indices of the singularities, and identify the local contributions with twisted sectors of age less than one. We determine how categorical genus changes under general \(\mathbb Q\)-Gorenstein deformation and express the resulting decrease in terms of the smoothable \(T\)-content. When \(X\) admits a toric \(\mathbb Q\)-Gorenstein degeneration, the categorical genus of a general deformation agrees with Tveiten's mutable genus, given by the residual interior lattice points of the corresponding Fano polygon. Under additional hypotheses, we recover this genus formula through the singularities of compactified mirror curves.
\end{abstract}



\section{Introduction}

For a smooth del Pezzo surface $Y$, the antisymmetrization of the numerical Euler pairing has rank $2$, as follows from Hirzebruch--Riemann--Roch in Kuznetsov's numerical coordinates \cite[Example 3.5]{Kuz} (Corollary~\ref{ex:smooth-del-pezzo}).  On the mirror side, a general Landau--Ginzburg fiber has a smooth projective model of genus one \cite{galkin2008delpezzo,AKO}.  This agreement suggests that half the rank of the antisymmetric Euler pairing may behave like a genus.

The motivating singular example is the general hypersurface $X_{10}\subset\Pj(1,2,3,5)$, which has one point of type $\frac13(1,1)$ \cite[Section 7.3, entry 11]{OP}.  The skew Euler form of its canonical stack has rank $4$, matching the genus-two mirror curve.  The author's earlier paper constructs a full strong exceptional collection of 13 line bundles and, in the one-point cascade, computes its Euler matrix, whose antisymmetric part has rank $4$ \cite[Theorem~1.4(i), Proposition~7.2]{GS}.  Sergey Galkin suggested that this rank should reflect mirror genus, motivating the present study.

For a quotient surface $X$, its canonical stack $\cX$ is the smooth Deligne--Mumford stack whose coarse moduli space is $X$.  Let $\chi_{\cX}$ denote the numerical Euler pairing on the numerical Grothendieck group of a smooth proper surface stack $\cX$, and write
\[
\chi^-_{\cX}(v,w)=\chi_{\cX}(v,w)-\chi_{\cX}(w,v).
\]
We call $g_{\mathrm{cat}}(\cX):=\frac12\rk_{\Q}\chi^-_{\cX}$ the \emph{categorical genus}.  The Corti--Heuberger families show that this invariant changes with the number of $\frac13(1,1)$ points, each of which contributes two to the rank by the special McKay calculation.  The weighted-projective stacks $\Pj(1,1,n)$ show that the contribution also depends on the cyclic quotient type.  These examples lead to a formula for an arbitrary basket of cyclic quotient singularities.

\begin{theorem}[Categorical genus for cyclic quotient surfaces]\label{thm:intro-cyclic}
Let $X$ be a log del Pezzo surface whose singularities are of type $\frac1{n_j}(1,q_j)$, $j=1,\dots,s$, and let $\cX$ be its canonical smooth Deligne--Mumford stack.  For each point, define its width $w_j=\gcd(n_j,q_j+1)$ and its local Gorenstein index $\ell_j=n_j/w_j$.  The rank of the skew Euler form is
\[
\rk_{\Q}\chi^-_{\cX}
=2+\sum_{j=1}^s\bigl(n_j-\gcd(n_j,q_j+1)\bigr).
\]
Consequently, the categorical genus is given by
\[
g_{\mathrm{cat}}(\cX)
=1+\frac12\sum_{j=1}^s w_j(\ell_j-1).
\]
\end{theorem}

The basket formula also has a direct interpretation in terms of the inertia stack of $\cX$.  Let $H^{\operatorname{age}<1}_{\cX}$ denote the vector space spanned by the nonidentity twisted sectors of $\cX$ whose age is strictly less than one, as defined in Definition~\ref{def:age-twisted-sectors}.  The following theorem records the resulting age count.

\begin{theorem}[Age interpretation]\label{thm:intro-age}
For the canonical stack $\cX$ of a log del Pezzo surface with cyclic quotient singularities, one has
\[
g_{\mathrm{cat}}(\cX)=1+\dim H^{\operatorname{age}<1}_{\cX}.
\]
\end{theorem}

The age formula counts the sectors of the original canonical stack.  A direct comparison with a mirror curve can therefore fail when the surface has $\Q$-Gorenstein smoothable singularities.  To describe the mirror side, recall that a Fano polygon is a lattice polygon $P\subset N_{\mathbb R}$ in a rank-two lattice $N$, with primitive vertices and the origin strictly in its interior.  Its spanning fan, whose two-dimensional cones are the cones over the edges of $P$, defines a toric del Pezzo surface $X_P$.

For a cyclic quotient point of width $w$ and local Gorenstein index $\ell$, write
\[
w=m\ell+w_0,\qquad 0\leq w_0<\ell.
\]
Here $m$ counts the primitive $T$-cones in a crepant subdivision of the spanning fan, and $w_0$ is the width of the residual cone.  The $T$-part is $\Q$-Gorenstein smoothable, while the residual cone gives the residual cyclic quotient singularity, or no singular point if $w_0=0$.  Definition~\ref{def:residual-basket} formalizes the residual basket of a surface.  For a Fano polygon, the residual basket appearing in its singularity content is obtained from its residual cones.

Definition~\ref{def:residual-lattice-points} associates a residual lattice-point set $R$ to a crepant subdivision of the spanning fan.  The set may depend on the subdivision, but the number of its points in the interior of $P$ does not.  Tveiten's mutable genus is
\[
g_{\mathrm{mut}}(P)=1+\#(R\cap P^\circ).
\]

Definition~\ref{def:mmlp-parameter-space} records the maximal-mutability condition and the distinguished affine parameter space $L_P^T$ with prescribed $T$-binomial edge coefficients.  Tveiten proves that, for a general maximally mutable polynomial $f$, the smooth projective model of a general fiber $f=\eta$ has genus $g_{\mathrm{mut}}(P)$ and that this genus is invariant under mutation \cite[Theorem 3.13]{Tveiten}.  We use the formulation in \cite[Theorem 3.13]{CKPT}.

The weighted projective plane $\Pj(1,1,4)$ exhibits the obstruction to comparing this mirror genus with the categorical genus of the original stack.  Its canonical stack has a point of type $\frac14(1,1)$ with $w=2$ and $\ell=2$, so Theorem~\ref{thm:intro-cyclic} gives categorical genus two.  This point is a $T$-singularity with $m=1$ and $w_0=0$, so it disappears under a sufficiently general $\Q$-Gorenstein deformation.  The corresponding Fano polygon has no residual interior lattice points, and its general maximally mutable fiber has genus one.  In contrast, the locally rigid $\frac13(1,1)$ points in the Corti--Heuberger families retain their full contribution.  The comparison therefore requires the categorical genus of the deformed canonical stack, rather than that of the original stack, when smoothable singularities are present.

Let $\mathfrak X\to(B,0)$ be an analytic miniversal $\Q$-Gorenstein deformation of a log del Pezzo surface $X$ with cyclic quotient singularities.  For a sufficiently general parameter $t$, let $\mathcal X_t$ be the canonical stack of the fiber $X_t$.  The categorical genus under general $\Q$-Gorenstein deformation is
\[
g_{\mathrm{cat}}^{\mathrm{qG}}(X):=g_{\mathrm{cat}}(\mathcal X_t),
\]
as formalized in Definition~\ref{def:qg-generic-genus}.  It is computed from the residual basket left after the smoothable $T$-parts disappear.  In the formula below, $m_j$ counts the primitive $T$-cones at the $j$th singular point and $w_{0,j}$ is its residual width.

For the mirror comparison, suppose that $X$ admits a toric $\Q$-Gorenstein degeneration with special fiber $X_{P_0}$.  The residual basket of $P_0$ agrees with that of a sufficiently general $\Q$-Gorenstein deformation of $X$.  Indeed, residual baskets are preserved under $\Q$-Gorenstein deformation and polygon mutation \cite[Theorem 3 and its proof]{ACC+}.  Thus the same basket is obtained from $P_0$ or any mutation-equivalent Fano polygon.

\begin{theorem}[Categorical genus under deformation and mirror genus]\label{thm:intro-mirror}
Let $X$ be a log del Pezzo surface with cyclic quotient singularities and canonical stack $\cX$.  At each singular point, use the widths and local Gorenstein indices of Theorem~\ref{thm:intro-cyclic} and write
\[
w_j=m_j\ell_j+w_{0,j},\qquad 0\le w_{0,j}<\ell_j.
\]
The value $g_{\mathrm{cat}}^{\mathrm{qG}}(X)$ is independent of the sufficiently general fiber and the chosen analytic miniversal deformation, and satisfies
\[
\begin{aligned}
g_{\mathrm{cat}}^{\mathrm{qG}}(X)
&=1+\frac12\sum_j w_{0,j}(\ell_j-1),\\
g_{\mathrm{cat}}(\cX)-g_{\mathrm{cat}}^{\mathrm{qG}}(X)
&=\sum_j m_j\binom{\ell_j}{2}.
\end{aligned}
\]
If moreover $X$ admits a toric $\Q$-Gorenstein degeneration to $X_{P_0}$, let $P$ be $P_0$ or a mutation-equivalent Fano polygon.  For a general maximally mutable Laurent polynomial $f$ with Newton polygon $P$ and a general value $\eta$, the following equalities hold:
\[
g_{\mathrm{cat}}^{\mathrm{qG}}(X)
=g_{\mathrm{mut}}(P)
=1+\#(R\cap P^\circ)
=g\bigl(\widetilde{\{f=\eta\}}\bigr),
\]
where the tilde denotes the smooth projective model of the fiber.
\end{theorem}

The dimension comparison in this theorem also places the result in relation to Conjecture~B of \cite{ACC+}.  The affine parameter space $L_P^T$ is defined in Definition~\ref{def:mmlp-parameter-space}.  Its residual-point parametrization gives $\dim L_P^T=\#(R\cap P^\circ)$ \cite[Proposition 3.7]{CKPT}.  Conjecture~B predicts an affine-linear map from this parameter space to the space spanned by the twisted sectors of age less than one for the canonical stack of a sufficiently general $\Q$-Gorenstein deformation.  Its target is $H^{\operatorname{age}<1}_{\cX_t}$, not in general $H^{\operatorname{age}<1}_{\cX}$ for the original stack, whose smoothable sectors may disappear.  The numerical equality identifies the dimensions of these spaces, but does not prove the predicted map or its period-matching assertion.

Under the simple-root, nonvanishing, and irreducibility hypotheses of Proposition~\ref{prop:boundary-delta}, we independently compute the boundary singularities: each edge $E$ with $T$-part multiplicity $m_E$ and height $\ell_E$ gives an $\ell_E$-branch singularity with delta invariant $m_E\binom{\ell_E}{2}$.  For the toric surface $X_P$, normalization identifies the sum of these invariants with the genus loss from the compactified curve to its smooth projective model.  This realizes the $T$-content drop in Theorem~\ref{thm:intro-mirror} geometrically.  Tveiten describes the defect when mutation factors coincide \cite[Remark 3.14, arXiv version 2]{Tveiten}, and his general genus formula does not require the edgewise hypotheses used here.

If $X$ is locally $\Q$-Gorenstein rigid, none of its singularities has a smoothable $T$-part, so the original and residual categorical genera agree.  The first corollary records the resulting mirror comparison.

\begin{corollary}[The locally $\Q$-Gorenstein rigid case]\label{cor:intro-rigid}
If $X$ is locally $\Q$-Gorenstein rigid and admits a toric $\Q$-Gorenstein degeneration, then the three genera agree:
\[
g_{\mathrm{cat}}(\cX)
=g_{\mathrm{cat}}^{\mathrm{qG}}(X)
=g\bigl(\widetilde{\{f=\eta\}}\bigr).
\]
\end{corollary}

For a toric del Pezzo surface, the full interior lattice-point count instead controls the categorical genus and the genus of a general Laurent-polynomial fiber.  Here the Laurent polynomial is general among all polynomials with Newton polygon $P$, not only among maximally mutable ones.

\begin{corollary}[The toric case]\label{cor:intro-toric}
Let $P\subset N_{\mathbb R}$ be a Fano polygon and let $\mathcal X_P$ be the canonical stack of the toric del Pezzo surface defined by the fan over the faces of $P$.  Then the categorical genus is given by
\[
g_{\mathrm{cat}}(\mathcal X_P)=\#(P^\circ\cap N).
\]
Consequently, for a general Laurent polynomial $h$ with Newton polygon $P$ and a general value $\eta$, the smooth projective model of the fiber satisfies
\[
g_{\mathrm{cat}}(\mathcal X_P)=g\bigl(\widetilde{\{h=\eta\}}\bigr).
\]
\end{corollary}

These comparisons are motivated by homological mirror symmetry but do not depend on it.  Harder--Thompson relate the numerical $K$-theory and Euler pairing of a genus-one Lefschetz fibration to relative homology and the Seifert pairing \cite[Sections 4.2 and 5]{HarderThompson}.  The antisymmetrization of the Seifert pairing is, up to sign, the pullback of the fiber intersection form.  When vanishing cycles span the nondegenerate genus part of the fiber homology, the resulting skew form has rank $2g$ for a genus-$g$ fiber.  Simeonov proves homological mirror symmetry for effective orbifold log Calabi--Yau surfaces at the large complex structure limit under additional hypotheses on the anticanonical cycle and the complex structure of its minimal resolution \cite[Definition 1.1 and Theorem 0.1]{Simeonov}.  His theorem allows general cyclic quotient types, while the numerical formula here applies to all cyclic quotient log del Pezzo surfaces without an HMS assumption.

Section~\ref{sec:categorical} develops the geometric model, and Section~\ref{sec:serre} proves the cyclic basket formula.  Section~\ref{sec:age} gives its age and lattice-point interpretations.  Sections~\ref{sec:mirror} and~\ref{sec:qg-generic} establish the residual comparison and the independent boundary calculation under its additional hypotheses.  Section~\ref{sec:explicit-mirrors} presents the explicit curve and polygon examples, including the corrected Petracci family.  The period comparisons for the Oneto--Petracci examples remain conditional on the orbifold Quantum Lefschetz and Abelian/non-Abelian conjectures \cite[Theorem 2.2 and Conjectures 5.2 and 6.1]{OP}.  Appendix~\ref{app:numerical-inertia} supplies the numerical inertial Riemann--Roch argument.

\section{Categorical genus and the geometric model}\label{sec:categorical}

We begin by isolating the smooth contribution to the antisymmetric Euler pairing, which provides the baseline for the singular calculations.

Throughout, the ground field is $\C$.  All stacks considered are separated Deligne--Mumford stacks of finite type, hence tame.  The geometric setting is the same as in the author's earlier study of exceptional collections on canonical stacks of log del Pezzo surfaces with $\frac13(1,1)$ singularities \cite{GS}.

\begin{definition}[Log del Pezzo surfaces and canonical stacks]
A log del Pezzo surface is a normal projective surface $X$ with quotient singularities and ample $-K_X$.  Its basket of singularities is the multiset of its singularity types, counted with multiplicity.  Its canonical stack $\cX$ is the smooth Deligne--Mumford stack with coarse moduli space $X$ and no stabilizers away from the singular points.
\end{definition}

\begin{definition}[Euler pairing and numerical Grothendieck group]\label{def:euler-pairing}
For a smooth proper Deligne--Mumford stack $\mathcal Z$, the Euler pairing and its antisymmetric part are
\[
\chi_{\mathcal Z}(E,F)
=\sum_i(-1)^i\dim\Ext^i(E,F),
\qquad
\chi_{\mathcal Z}^-(E,F)
:=\chi_{\mathcal Z}(E,F)-\chi_{\mathcal Z}(F,E).
\]
We use $K_0^{\mathrm{num}}(\mathcal Z)=K_0(\mathcal Z)/\operatorname{rad}\chi_{\mathcal Z}$, where $\operatorname{rad}\chi_{\mathcal Z}=\{v\in K_0(\mathcal Z):\chi_{\mathcal Z}(v,u)=0\text{ for every }u\in K_0(\mathcal Z)\}$.
\end{definition}

Here $K_0(\mathcal Z)$ denotes $K_0(\Db(\coh\mathcal Z))$, identified with $G_0(\mathcal Z)$ by the alternating sum of coherent cohomology sheaves.  The group $K_0(\mathcal Z)$ is integral, and subscripts indicate scalar extension.  Numerical ranks in this section are taken over $\Q$ unless another coefficient field is displayed.

The left and right radicals agree.  Write $R_L$ and $R_R$ for these radicals.  The Serre functor is an autoequivalence preserving the Euler pairing, so it preserves $R_L$.  Serre duality gives $\chi(v,u)=\chi(u,Sv)$, hence $S(R_L)\subseteq R_R$.  The two radicals have equal codimension because the Euler pairing has finite rank.  Thus $R_L=R_R$, and $S$ preserves their common radical.  In particular, the Serre functor induces an automorphism of $K_0^{\mathrm{num}}(\mathcal Z)$.

For a smooth projective surface $Y$, Kuznetsov \cite[Example 3.5]{Kuz} describes the numerical group as
\[
K_0^{\mathrm{num}}(Y)_{\Q}
\simeq
\Q\oplus\NS(Y)_{\Q}\oplus\Q,
\]
with coordinates represented by rank, first Chern class, and a degree-two class.  Hirzebruch--Riemann--Roch shows that only the first two coordinates enter the skew form.

\begin{lemma}[The smooth-surface skew form]\label{lem:RR-skew}
For $A,B\in K_0(Y)_{\Q}$, write $r_A=\rk(A)$ and $D_A=c_1(A)$, with analogous notation for $B$.  Their skew Euler pairing is
\begin{equation}\label{eq:smooth-skew}
\chi_Y^-(A,B)
=-K_Y\cdot\bigl(r_A D_B-r_BD_A\bigr).
\end{equation}
\end{lemma}

\begin{proof}
Hirzebruch--Riemann--Roch gives
\[
\chi_Y(A,B)=\int_Y
\operatorname{ch}(A^\vee)\operatorname{ch}(B)\operatorname{td}(Y).
\]
The degree-two part of $\operatorname{ch}(A^\vee)\operatorname{ch}(B)$ is the symmetric expression
\[
r_A\operatorname{ch}_2(B)+r_B\operatorname{ch}_2(A)-D_A D_B.
\]
The contribution involving $\operatorname{td}_2(Y)$ is also symmetric, so both cancel upon antisymmetrization.  The remaining term comes from $\operatorname{td}_1(Y)=-K_Y/2$.  The degree-one component of the difference
\[
\operatorname{ch}(A^\vee)\operatorname{ch}(B)
-\operatorname{ch}(B^\vee)\operatorname{ch}(A)
\]
is $2(r_A D_B-r_BD_A)$.  Multiplying this by $\operatorname{td}_1(Y)$ and integrating gives \eqref{eq:smooth-skew}.
\end{proof}

\begin{corollary}[The smooth del Pezzo rank]\label{ex:smooth-del-pezzo}
Every smooth del Pezzo surface $Y$ satisfies
\begin{equation}\label{eq:smooth-rank-two}
\rk_{\Q}\chi_Y^-=2.
\end{equation}
\end{corollary}

\begin{proof}
In the rational coordinates $(r,D,c)=(\rk,c_1,\deg\operatorname{ch}_2)$, Lemma~\ref{lem:RR-skew} gives
\[
\chi_Y^-\bigl((r,D,c),(r',D',c')\bigr)
=-K_Y\cdot(rD'-r'D).
\]
Thus the degree-two coordinate lies in the radical, and the form has rank at most two because it depends only on the two linear functionals $r$ and $K_Y\cdot D$.  On the span of $e_0=(1,0,0)$ and $e_K=(0,K_Y,0)$ its matrix is
\[
\begin{pmatrix}0&-K_Y^2\\K_Y^2&0\end{pmatrix}.
\]
Since $K_Y^2$ is positive, this restriction is nondegenerate and the upper bound is attained.
\end{proof}

\begin{definition}[Categorical genus]\label{def:categorical-genus}
For a smooth proper surface stack $\cX$, its \emph{categorical genus} is
\[
\boxed{
g_{\mathrm{cat}}(\cX)
:=\frac12\rk_{\Q}\chi^-_{\cX},}
\]
where the pairing is taken on $K_0^{\mathrm{num}}(\cX)_{\Q}$.  The skew form factors through this numerical quotient, so the rank agrees with that computed on $K_0(\cX)_{\Q}$.
\end{definition}

The rank of the skew Euler form is even, so the definition is integral.  An exact equivalence preserves the Euler pairing and its numerical quotient, which makes categorical genus an invariant of the $\C$-linear triangulated category.  A full exceptional collection provides a matrix representation of the pairing.  Corollary~\ref{ex:smooth-del-pezzo} gives $g_{\mathrm{cat}}(Y)=1$ for a smooth del Pezzo surface.  In the log del Pezzo setting, $g_{\mathrm{cat}}-1$ counts the additional two-dimensional nondegenerate summands of the skew form beyond this smooth rank-two contribution.

\begin{remark}[Choice of stack]\label{rem:choice-of-stack}
The value depends on the chosen stack, not only on its coarse moduli space.  We use the canonical stack in the basket formulas.  Adding root-stack structure can change the inertia and the Euler pairing, so the formulas do not automatically apply to other stacks with the same coarse moduli space.  We do not apply Definition~\ref{def:categorical-genus} to the singular coarse moduli space.  The residual categorical genus defined in Section~\ref{sec:qg-generic} uses the canonical stack of a general deformation and need not equal the invariant of the original canonical stack.
\end{remark}

We now compute the $\frac13(1,1)$ case using special McKay.  This independent geometric argument expresses the rank through intersection pairings and connects the invariant with exceptional collections.

For this calculation, let $X$ be a log del Pezzo surface with singular points $p_1,\dots,p_r$, all of this type.  Write $f:Y\to X$ for the minimal resolution and $C_i\subset Y$ for the exceptional curve over $p_i$.  The exceptional curves satisfy
\[
C_i\cdot C_j=-3\delta_{ij},
\]
and the canonical classes are related by the discrepancy formula
\begin{equation}\label{eq:discrepancy}
K_Y=f^*K_X-\frac13\sum_{i=1}^r C_i.
\end{equation}
The global cyclic case of the special McKay correspondence \cite[Proposition 8.1]{IU} gives a fully faithful functor $\Phi:\Db(\coh Y)\to\Db(\coh\cX)$ and a semiorthogonal decomposition
\begin{equation}\label{eq:IU-sod}
\Db(\coh\cX)
=\left\langle E_1,\dots,E_r,\Phi\Db(\coh Y)\right\rangle.
\end{equation}
The functor $\Phi$ is the integral functor with kernel $\mathcal O_{(Y\times_X\cX)_{\mathrm{red}}}$.

The \emph{special representations} are the trivial representation and those associated with exceptional curves on the minimal resolution. The others are called non-special \cite[Section 2.3]{GR}.

In \eqref{eq:IU-sod}, the exceptional object $E_i$ is the sheaf $(\iota_i)_*\rho_2$ supported on the residual gerbe $\iota_i:\mathcal B\mu_3\hookrightarrow\cX$, with the character convention specified below.  Objects at distinct points are orthogonal because their supports are disjoint, and the order in \eqref{eq:IU-sod} gives $R\!\operatorname{Hom}(\Phi A,E_i)=0$.

\begin{remark}[Character conventions]\label{rem:character-conventions}
In \cite[Section 2.3 and Example 2.7]{GR}, a generator $h$ has cotangent weights $(\zeta,\zeta)$ and $\rho_b(h)=\zeta^b$.  The special representations are $\rho_0,\rho_1$, so the complement uses $\rho_2$.  For the generator $g=h^{-1}$ acting on tangent vectors with weights $(\zeta,\zeta)$, the character value is $\rho_2(g)=\zeta^{-2}=\zeta$.  This specifies the sheaf, not just its support, and fixes the adjoint calculation below.
\end{remark}

For the left adjoint $\Psi$ of $\Phi$, the local calculation of Gugiatti--Rota \cite[Example 3.5]{GR}, specialized to $n=3$ and representation index $2$, gives
\[
\Psi(E_i)\simeq\mathcal O_{C_i}(-2).
\]
\begin{lemma}[Euler pairings with residual-gerbe sheaves]\label{lem:local-pairing}
For every $A\in K_0(Y)_{\Q}$, one has
\begin{equation}\label{eq:local-pairing}
\chi_{\cX}(E_i,\Phi A)=-c_1(A)\cdot C_i,
\qquad
\chi_{\cX}(\Phi A,E_i)=0.
\end{equation}
\end{lemma}

\begin{proof}
Semiorthogonality in \eqref{eq:IU-sod} gives the second equality.  For the first, adjunction yields
\[
\chi_{\cX}(E_i,\Phi A)=\chi_Y(\mathcal O_{C_i}(-2),A).
\]
Since locally free sheaves generate $K_0(Y)$, it suffices by additivity to take $A$ to be a vector bundle.  For the Cartier-divisor inclusion $j:C_i\hookrightarrow Y$, applying $R\mathcal{H}om_Y(-,A)$ to the Cartier divisor triangle
\[
\mathcal O_Y(-C_i)\longrightarrow\mathcal O_Y\longrightarrow j_*\mathcal O_{C_i},
\]
gives
\[
R\mathcal{H}om_Y(j_*\mathcal O_{C_i},A)
\simeq j_*\bigl(A|_{C_i}\otimes\mathcal O_{C_i}(C_i)\bigr)[-1].
\]
By the defining adjunction for $j^!$, this identifies
\[
j^!A=A|_{C_i}\otimes\mathcal O_{C_i}(C_i)[-1]
=A|_{C_i}\otimes\mathcal O_{C_i}(-3)[-1].
\]
Taking $R\!\operatorname{Hom}_{C_i}(\mathcal O_{C_i}(-2),j^!A)$ therefore gives
\[
R\!\operatorname{Hom}_Y(\mathcal O_{C_i}(-2),A)
\simeq
R\Gamma\bigl(C_i,A|_{C_i}\otimes\mathcal O_{C_i}(-1)\bigr)[-1].
\]
Let $\delta_A$ denote the degree $c_1(A)\cdot C_i$ of $A|_{C_i}$. Riemann--Roch on $C_i\simeq\Pj^1$ gives
\[
\chi\bigl(C_i,A|_{C_i}(-1)\bigr)=\delta_A.
\]
The shift $[-1]$ changes the sign of the Euler characteristic, proving \eqref{eq:local-pairing}.
\end{proof}

The smooth and local pairing formulas can now be assembled into a block matrix.  The minimal resolution $Y$ is rational: quotient singularities are rational, Kawamata--Viehweg vanishing gives $H^1(Y,\mathcal O_Y)=H^1(X,\mathcal O_X)=0$, and $-K_Y$ is big, so $H^0(Y,2K_Y)$ vanishes.  Castelnuovo's criterion applies, as explained in \cite[Remark 4.3]{GR}.  We use a rational version of Kuznetsov's coordinates \cite[Example 3.5]{Kuz}, replacing the normalized Euler-characteristic coordinate by $\deg\operatorname{ch}_2$.  Explicitly, set
\[
\kappa:K_0(Y)_{\Q}\longrightarrow
\Q\oplus\NS(Y)_{\Q}\oplus\Q,
\qquad
[A]\longmapsto
\bigl(\rk(A),c_1(A),\deg\operatorname{ch}_2(A)\bigr).
\]
The rational Chern-character isomorphism described by Fulton \cite[Chapter 15]{Fulton} gives
\[
\operatorname{ch}_{\Q}:K_0(Y)_{\Q}\xrightarrow{\sim}
\operatorname{CH}^*(Y)_{\Q}.
\]
Since $Y$ is rational, its Chow groups are
\[
\operatorname{CH}^0(Y)_{\Q}\cong\Q,
\qquad
\operatorname{CH}^1(Y)_{\Q}\cong\NS(Y)_{\Q},
\qquad
\operatorname{CH}^2(Y)_{\Q}\cong\Q.
\]
Thus $\kappa$ is an isomorphism.  Lemma~\ref{lem:RR-skew} shows that its third coordinate lies in the radical of the smooth skew form, while Lemma~\ref{lem:local-pairing} shows that it contributes nothing to the cross-pairings.

The semiorthogonal decomposition \eqref{eq:IU-sod} induces
\[
K_0(\cX)_{\Q}
\simeq
\bigoplus_{i=1}^r\Q[E_i]\oplus K_0(Y)_{\Q}.
\]
Using this decomposition and the coordinates $\kappa$, define
\[
e_0:=\Phi_*\bigl(\kappa^{-1}(1,0,0)\bigr),
\qquad
e_i:=[E_i]\quad(1\le i\le r).
\]
The rank and residual-gerbe classes span the subspace
\[
U:=\Q e_0\oplus\Q[E_1]\oplus\cdots\oplus\Q[E_r]
\subset K_0(\cX)_{\Q}.
\]
The rank of the skew Euler form can therefore be computed on $U\oplus\NS(Y)_{\Q}$, with the degree-two radical omitted.

\begin{proposition}[Rank from intersection pairings]\label{prop:rank-span}
The skew Euler form satisfies
\[
\rk_{\Q}\chi^-_{\cX}
=2\dim_{\Q}\Span\{K_Y,C_1,\dots,C_r\}.
\]
\end{proposition}

\begin{proof}
Lemmas~\ref{lem:RR-skew} and~\ref{lem:local-pairing}, together with the mutual orthogonality of the $E_i$, show that $U$ and $\NS(Y)_{\Q}$ are both isotropic.  Relative to their direct sum, the skew form has matrix
\[
\begin{pmatrix}
0&B\\
-B^{\mathsf T}&0
\end{pmatrix},
\]
where $B:\NS(Y)_{\Q}\to U^\vee$ is defined by $B(D)(u)=\chi^-_{\cX}(u,D)$ and hence is given by
\[
B(D)=-(K_Y\cdot D)e_0^\vee
-\sum_{i=1}^r(C_i\cdot D)e_i^\vee.
\]
Indeed, the pairings are $\chi^-_{\cX}(e_0,D)=-K_Y\cdot D$ and $\chi^-_{\cX}(E_i,D)=-C_i\cdot D$. Reversing the arguments gives the positive signs in the lower-left block.  Thus the rows of $B$ are the negatives of the intersection functionals associated with $K_Y,C_1,\dots,C_r$.  The nondegenerate intersection pairing induces an isomorphism
\[
\NS(Y)_{\Q}\xrightarrow{\sim}\NS(Y)_{\Q}^\vee,
\qquad
L\longmapsto\bigl(D\longmapsto L\cdot D\bigr),
\]
and the ranks consequently satisfy
\[
\rk(B)=\dim_{\Q}\Span\{K_Y,C_1,\dots,C_r\}.
\]
The image of the block matrix is $\operatorname{Im}(B)\oplus\operatorname{Im}(B^{\mathsf T})$, so its rank is $2\rk(B)$.
\end{proof}

It remains to determine the dimension of this span.  The discrepancy formula and the positivity of $K_X^2$ give the required independence.

\begin{proposition}[The $\frac13(1,1)$ calculation]\label{prop:third-rank}
If $X$ has $r$ singularities of type $\frac13(1,1)$, then the rank of its skew Euler form and its categorical genus are given by
\[
\boxed{
\rk_{\Q}\chi^-_{\cX}=2r+2,
\qquad
g_{\mathrm{cat}}(\cX)=r+1}.
\]
\end{proposition}

\begin{proof}
The curves $C_i$ are linearly independent because their intersection matrix is $-3I_r$.  Suppose that $K_Y$ belonged to their span.  The discrepancy formula \eqref{eq:discrepancy} would place $f^*K_X$ in the same span, while the projection formula gives
\[
f^*K_X\cdot C_i=0
\]
for every $i$.  The intersection form on $\Span\{C_1,\dots,C_r\}$ is nondegenerate, so this would force $f^*K_X=0$, contradicting
\[
(f^*K_X)^2=K_X^2>0.
\]
Therefore $K_Y,C_1,\dots,C_r$ are linearly independent, and Proposition~\ref{prop:rank-span} gives the result.
\end{proof}

\begin{example}[$X_{10}$]\label{ex:x10-rank}
The motivating hypersurface $X_{10}\subset\Pj(1,2,3,5)$ has one singularity of type $\frac13(1,1)$ \cite[Section 7.3, entry 11]{OP}.  It is the case $r=1$ of Proposition~\ref{prop:third-rank}, which gives
\[
\rk_{\Q}\chi^-_{\cX_{10}}=4,
\qquad
g_{\mathrm{cat}}(\cX_{10})=2.
\]
\end{example}

\section{The Serre operator and cyclic quotient singularities}\label{sec:serre}

The preceding section computes the special case of singularities of type $\frac13(1,1)$ by special McKay.  This section extends the calculation to arbitrary cyclic quotient types, using the weighted-projective stacks $\mathcal P_n=\Pj(1,1,n)$ as an introductory model.  The Serre operator and numerical inertial Riemann--Roch then give the general rank formula and its derived-equivalence consequences.

\begin{proposition}[The weighted-projective calculation]\label{prop:weighted-rank}
For every $n\ge2$, the skew Euler form on $\mathcal P_n$ has rank
\[
\boxed{
\rk_{\Q}\chi^-_{\mathcal P_n}
=n+2-\gcd(n,2)
=2+\bigl(n-\gcd(n,2)\bigr)}.
\]
Its categorical genus is
\[
\boxed{g_{\mathrm{cat}}(\mathcal P_n)=\left\lfloor\frac{n+1}{2}\right\rfloor}.
\]
\end{proposition}

\begin{proof}
With $t=[\mathcal O(-1)]$, the toric $K$-theory presentation of Borisov--Horja \cite[Theorem 4.10]{BorisovHorja} specializes to
\[
K_0(\mathcal P_n)_{\Q}
\simeq
A_n:=\frac{\Q[t,t^{-1}]}{(1-t)^2(1-t^n)}.
\]
Indeed, the three torus-invariant divisor line bundles are $\mathcal O(1)$, $\mathcal O(1)$, and $\mathcal O(n)$.  Their classes are $t^{-1},t^{-1},t^{-n}$, and the resulting relation $(1-t^{-1})^2(1-t^{-n})=0$ differs from the displayed relation by an invertible Laurent monomial.
Because the defining polynomial has nonzero constant term, $t$ is already invertible in the ordinary polynomial quotient, so the dimension of $A_n$ is $n+2$.  The standard full collection
\[
\bigl(\mathcal O(-n-1),\mathcal O(-n),\dots,\mathcal O\bigr)
\]
has an upper-triangular Euler matrix by semiorthogonality, with diagonal entries one by exceptionality. Its determinant is therefore one, so the Euler form is nondegenerate on this $(n+2)$-dimensional space.

Since the canonical bundle $\omega_{\mathcal P_n}$ is isomorphic to $\mathcal O(-n-2)$, the Serre operator $S$ acts on $A_n$ by multiplication by $t^{n+2}$.  Serre duality gives
\[
\chi^-(v,w)=\chi\bigl(v,(1-S)w\bigr),
\]
so nondegeneracy of $\chi$ implies
\[
\rk\chi^-=\rk(1-S).
\]
The rank calculation therefore reduces to finding the kernel of multiplication by $1-t^{n+2}$.  In a polynomial quotient $\Q[t]/(F)$, multiplication by $G$ has kernel of dimension $\deg\gcd(F,G)$.  Indeed, if $D=\gcd(F,G)$, then $F$ divides $Gh$ precisely when $F/D$ divides $h$, so the kernel consists of the multiples of $F/D$ modulo $F$ and has dimension $\deg D$.  Equivalently, the Chinese remainder theorem sums these kernel dimensions over the primary factors of $F$, including repeated roots.  In our case, the relevant polynomials are
\[
F=(1-t)^2(1-t^n),
\qquad
G=1-t^{n+2}.
\]
Put $\delta=\gcd(n,2)$.  Since the greatest common divisor of $n$ and $n+2$ is $\delta$, the common roots of $1-t^n$ and $1-t^{n+2}$ are exactly the $\delta$th roots of unity.  Although $F$ has a zero of order three at $t=1$, the polynomial $G$ has only a simple zero there.  Therefore the greatest common divisor is
\[
\gcd(F,G)=1-t^\delta
\]
up to a nonzero scalar.  Hence $\dim\ker(1-S)=\delta$, and subtracting from $\dim A_n=n+2$ proves the formula.
\end{proof}

The calculation separates the smooth rank-two contribution from the local correction $n-\gcd(n,2)$.  For a point of type $\frac1n(1,q)$, the canonical character on the sector indexed by $i$ has eigenvalue $\zeta^{-i(q+1)}$, where $\zeta$ is a primitive $n$th root of unity.  The following argument counts the sectors on which this character is nontrivial.

The Serre operator $S_{\cX}$ is the automorphism of the numerical Grothendieck group induced by the Serre functor.  For a smooth proper surface stack, this functor is $E\mapsto E\otimes\omega_{\cX}[2]$ \cite[Theorem 2.22]{NironiDuality}.  The shift acts trivially on $K_0$, so $S_{\cX}$ is tensor product with $\omega_{\cX}$.  The following lemma expresses the rank of the skew Euler form in terms of this automorphism.

\begin{lemma}[The skew Euler form and the Serre operator]\label{lem:skew-serre}
Let $\cX$ be a smooth proper Deligne--Mumford stack, and suppose that $V=K_0^{\mathrm{num}}(\cX)_{\C}$ is finite-dimensional with nondegenerate Euler pairing $\chi$.  If $S_{\cX}$ is induced by the Serre functor, then the two ranks are equal:
\begin{equation}\label{eq:skew-serre-rank}
\rk\chi^-=\rk(1-S_{\cX}).
\end{equation}
\end{lemma}

\begin{proof}
Serre duality and the same identity with $u$ and $v$ interchanged give
\[
\begin{aligned}
\chi(u,v)&=\chi(v,S_{\cX}u),\\
\chi^-(u,v)&=\chi\bigl(u,(1-S_{\cX})v\bigr).
\end{aligned}
\]
Thus the map $V\to V^\vee$ defined by the skew form is the composition of $1-S_{\cX}$ with the isomorphism $v\mapsto\chi(-,v)$.  Their ranks agree.
\end{proof}

To apply the lemma, we use the inertia stack to separate the contribution of the surface from those of the stabilizer characters.  We work over $\C$ to diagonalize these characters. Scalar extension does not change the ranks over $\Q$.

The inertia stack $I\cX$ parametrizes pairs $(x,g)$ with $g$ a stabilizer automorphism.  The identity component is the untwisted sector. The other connected components are the twisted sectors.

For a canonical stack with isolated cyclic quotient points $\frac1{n_j}(1,q_j)$, every nonidentity stabilizer element fixes only the corresponding point.  Since the stabilizers are abelian, the inertia stack decomposes as
\begin{equation}\label{eq:inertia-components}
I\cX=\cX\sqcup\coprod_j\coprod_{i=1}^{n_j-1}\mathcal B\mu_{n_j}^{(i)}.
\end{equation}
Each gerbe contributes one numerical Chow class after complexification.  To transfer this decomposition to the numerical Grothendieck group, we fix the Chern-character convention.

\begin{definition}[Inertial Chern character]
On the sector indexed by $g$, decompose the pullback of a vector bundle as $E|_{(g)}=\bigoplus_\lambda E_\lambda$, where $g$ acts on $E_\lambda$ by $\lambda$.  Our convention is
\begin{equation}\label{eq:inertial-character-convention}
\widetilde{\operatorname{ch}}(E)|_{(g)}
=\sum_\lambda\lambda\operatorname{ch}(E_\lambda).
\end{equation}
It evaluates characters at $g$, not at $g^{-1}$. On a point sector it is simply the trace of $g$ on the fiber.
\end{definition}
This agrees with Edidin's twisting operator \cite[Section 4.3.2]{Edidin}: on the inertia component represented by $g$, it weights each eigenbundle by its eigenvalue $\lambda=\lambda(g)$, so no inversion of the sector label occurs in \eqref{eq:inertial-character-convention}.

\begin{proposition}[Numerical inertial decomposition]\label{prop:numerical-inertia}
Let $X$ be a projective surface with isolated cyclic quotient singularities $\frac1{n_j}(1,q_j)$, and let $\cX$ be its canonical stack.  The inertial Chern character induces an $S_{\cX}$-invariant decomposition
\begin{equation}\label{eq:inertial-decomposition}
K_0^{\mathrm{num}}(\cX)_{\C}
\simeq A^*_{\mathrm{num}}(\cX)_{\C}
\oplus\bigoplus_j\bigoplus_{i=1}^{n_j-1}\C_{j,i}.
\end{equation}
The untwisted summand is the numerical Chow group of the smooth stack $\cX$, with its ordinary intersection product.  The line $\C_{j,i}$ is associated with the $i$th nonidentity element of $\mu_{n_j}$.  In these coordinates, $S_{\cX}$ is multiplication by $\widetilde{\operatorname{ch}}(\omega_{\cX})$.
\end{proposition}

Vistoli's rational comparison with the coarse moduli space \cite[Proposition 6.1]{Vistoli} will identify the untwisted rational Chow groups with those of $X$ for the calculation in Lemma~\ref{lem:untwisted-rank}.

The compatibility of the inertial Chern character with numerical equivalence is proved in Appendix~\ref{app:numerical-inertia}.  For the rank calculation, the consequence is that $1-S_{\cX}$ preserves the summands in \eqref{eq:inertial-decomposition}, and its ranks on them add.

The untwisted computation is the Serre-operator counterpart of Corollary~\ref{ex:smooth-del-pezzo}.  There Hirzebruch--Riemann--Roch computes the skew form directly on a smooth surface. Here the same rank-two contribution reappears as the untwisted part of $1-S_{\cX}$.

\begin{lemma}[The untwisted contribution]\label{lem:untwisted-rank}
For the canonical stack of a log del Pezzo surface $X$ with cyclic quotient singularities, the untwisted summand in \eqref{eq:inertial-decomposition} satisfies
\[
\rk\bigl((1-S_{\cX})|_{A^*_{\mathrm{num}}(\cX)_{\C}}\bigr)=2.
\]
\end{lemma}

\begin{proof}
We calculate on $A^*_{\mathrm{num}}(\cX)_{\C}$, where the inertial Chern character is the ordinary one.  Vistoli's rational pushforward to the coarse moduli space identifies the Chow groups of $\cX$ with those of $X$ \cite[Proposition 6.1]{Vistoli}. The intersection pairing remains the one on the smooth stack.  For rational Cartier divisors it agrees with the usual pairing on $X$ by the projection formula.  Quotient singularities are $\Q$-factorial: locally the norm of a defining equation upstairs makes a multiple of any Weil divisor Cartier.  Thus the rational class group and rational Picard group of $X$ are equal.  Moreover $H^1(X,\mathcal O_X)=H^2(X,\mathcal O_X)=0$: quotient singularities are rational and their minimal resolution is rational, as in \cite[Remark 4.3]{GR}.  In particular $\operatorname{Pic}^0(X)=0$, giving $\operatorname{Pic}(X)_{\Q}=\NS(X)_{\Q}$.  Through this correspondence, the degree-zero and degree-two numerical groups are each one-dimensional, while the divisor group is $\NS(X)_{\C}$.  Thus the numerical Chow group decomposes as
\[
A^*_{\mathrm{num}}(\cX)_{\C}\simeq\C\oplus\NS(X)_{\C}\oplus\C.
\]
Write a class as $(r,D,c)$, and let $K=K_X$.  Since the canonical stack has no stabilizers in codimension one, its coarse moduli map $\pi:\cX\to X$ satisfies $K_{\cX}=\pi^*K$ as rational divisor classes.  Multiplication by $e^K=1+K+K^2/2$ therefore gives
\[
S_{\cX}(r,D,c)
=\left(r,D+rK,c+K\cdot D+\frac{rK^2}{2}\right).
\]
The image of a class $(r,D,c)$ under $1-S_{\cX}$ is
\[
(1-S_{\cX})(r,D,c)
=\left(0,-rK,-K\cdot D-\frac{rK^2}{2}\right).
\]
The image is contained in the span of $(0,K,0)$ and $(0,0,1)$, so its dimension is at most two.  Since $-K_X$ is ample, $K^2$ is positive, and the two images
\[
\begin{aligned}
(1-S_{\cX})(1,0,0)&=\left(0,-K,-\frac{K^2}{2}\right),\\
(1-S_{\cX})(0,K,0)&=(0,0,-K^2)
\end{aligned}
\]
are linearly independent.  The untwisted contribution is therefore exactly two.
\end{proof}

The remaining summands are one-dimensional, so their contribution reduces to the canonical character at a single cyclic quotient point.

\begin{lemma}[The local cyclic contribution]\label{lem:local-cyclic-rank}
At a point of type $\frac1n(1,q)$, put $w=\gcd(n,q+1)$ and let $W=\bigoplus_{i=1}^{n-1}\C_i$ be its twisted summand in \eqref{eq:inertial-decomposition}.  The kernel of $(1-S_{\cX})|_W$ has dimension $w-1$, and its rank is $n-w$.
\end{lemma}

\begin{proof}
Fix a point of type $\frac1n(1,q)$ and a primitive $n$th root of unity $\zeta$.  The generator acts on the tangent fiber by $\operatorname{diag}(\zeta,\zeta^q)$, so its action on the canonical fiber $\bigwedge^2T^*$ is multiplication by $\zeta^{-(1+q)}$.  Thus, with convention \eqref{eq:inertial-character-convention}, $S_{\cX}$ acts on the component $\C_i$ indexed by $\zeta^i$ through
\[
\zeta^{-i(1+q)}.
\]
Using inverse elements to label the components would invert these eigenvalues without changing the rank.  Each $\C_i$ is one-dimensional, so it contributes one to $\rk(1-S_{\cX})$ unless its eigenvalue is $1$.  The latter condition is
\[
n\mid i(q+1).
\]
Put $w=\gcd(n,q+1)$ and write $n=wn'$ and $q+1=wa$, with $n'$ and $a$ coprime.  The following equivalences hold:
\[
n\mid i(q+1)\iff wn'\mid iwa\iff n'\mid i.
\]
Among $1\le i\le n-1$, the solutions to this divisibility condition are precisely
\[
i=kn',\qquad 1\le k\le w-1.
\]
Thus $\dim W=n-1$ and $\dim\ker((1-S_{\cX})|_W)=w-1$, giving rank $(n-1)-(w-1)=n-w$.  In particular, the following formula identifies the local contribution:
\[
\boxed{\text{local contribution of }\tfrac1n(1,q)=n-\gcd(n,q+1).}
\]
\end{proof}

For $q=1$, the condition $n\mid2i$ gives the same $w=\gcd(n,2)=\gcd(n,n+2)$ as the common-root count in Proposition~\ref{prop:weighted-rank}.  The root $t=1$ accounts for the untwisted kernel direction. The other $w-1$ roots correspond to fixed twisted sectors.  Thus the weighted-projective calculation is the $q=1$ model of the canonical-character count.

The following values recover the initial $\frac13(1,1)$ calculation and show that the correction depends on $q$, not just on the group order:
{\renewcommand{\arraystretch}{1.15}
\[
\begin{array}{c|c|c|c|c}
\text{singularity}&w=\gcd(n,q+1)&(w,\ell)&\substack{\text{local rank of the}\\\text{skew Euler form}}&\text{local genus}\\
\hline
\frac12(1,1)&2&(2,1)&0&0\\
\frac13(1,1)&1&(1,3)&2&1\\
\frac14(1,1)&2&(2,2)&2&1\\
\frac14(1,3)&4&(4,1)&0&0\\
\frac18(1,1)&2&(2,4)&6&3\\
\frac18(1,3)&4&(4,2)&4&2
\end{array}
\]}
The last two columns are the local contributions $n-w$ and $(n-w)/2$, respectively.  In particular, the $A_1$ singularity $\frac12(1,1)$ contributes zero to the rank of the skew Euler form and is invisible to categorical genus.

We now assemble these local and untwisted contributions.

\begin{theorem}[The general cyclic quotient formula]\label{thm:cyclic-rank}
Let $X$ be a log del Pezzo surface with cyclic quotient singularities $\frac1{n_j}(1,q_j)$, $j=1,\dots,s$, and let $\cX$ be its canonical stack.  Put $w_j=\gcd(n_j,q_j+1)$ and $\ell_j=n_j/w_j$.  The rank of its skew Euler form is
\[
\boxed{
\rk_{\Q}\chi^-_{\cX}
=2+\sum_{j=1}^s(n_j-w_j).}
\]
Consequently its categorical genus is
\[
\boxed{
g_{\mathrm{cat}}(\cX)
=1+\frac12\sum_{j=1}^s w_j(\ell_j-1).}
\]
\end{theorem}

\begin{proof}
By Lemma~\ref{lem:skew-serre}, the rank of the skew Euler form equals $\rk(1-S_{\cX})$.  Proposition~\ref{prop:numerical-inertia} splits this rank into the untwisted contribution, which is two by Lemma~\ref{lem:untwisted-rank}, and the local twisted contributions, which are $n_j-w_j$ by Lemma~\ref{lem:local-cyclic-rank}.  The resulting sum is unchanged by extension of scalars from $\Q$ to $\C$.  Dividing by two gives the genus formula.
\end{proof}

The argument also extends beyond log del Pezzo surfaces.  For any normal projective surface with isolated cyclic quotient singularities and $K_X\not\equiv0$, the formulas in Theorem~\ref{thm:cyclic-rank} remain valid.  In the untwisted computation, the numerical Chow group has the form $\C\oplus N^1(X)_{\C}\oplus\C$, and $1-S_{\cX}$ acts by
\[
(1-S_{\cX})(r,D,c)=\left(0,-rK_X,-K_X\cdot D-\frac{rK_X^2}{2}\right).
\]
Since $K_X\not\equiv0$, choose a curve $C$ with $K_X\cdot C\ne0$.  Quotient singularities are $\Q$-factorial, so the class of $C$ defines a divisor class $D$ in $N^1(X)_{\C}$ with this intersection.  The images of $(1,0,0)$ and $(0,D,0)$ are linearly independent, so the untwisted rank is two.  The local twisted-sector calculation is unchanged, proving both formulas.

Because a derived equivalence preserves both the numerical Grothendieck group and the antisymmetric Euler pairing, the cyclic quotient formula imposes constraints on the singularity baskets of equivalent canonical stacks.

\begin{corollary}[An obstruction to derived equivalence]\label{cor:derived-basket-obstruction}
Let $\mathcal X$ and $\mathcal X'$ be the canonical stacks of two log del Pezzo surfaces $X$ and $X'$ with cyclic quotient singularities.  If the $\C$-linear triangulated categories $\Db(\coh\mathcal X)$ and $\Db(\coh\mathcal X')$ are equivalent, then their categorical genera agree:
\[
g_{\mathrm{cat}}(\mathcal X)=g_{\mathrm{cat}}(\mathcal X').
\]
Writing $\rho(X)=\rk\NS(X)$, the singularity baskets satisfy
\[
\begin{aligned}
\sum_j(n_j-w_j)&=\sum_k(n'_k-w'_k),\\
\rho(X)+\sum_j(n_j-1)&=\rho(X')+\sum_k(n'_k-1).
\end{aligned}
\]
Failure of either equality therefore obstructs $\C$-linear derived equivalence of the canonical stacks.
\end{corollary}

\begin{proof}
A $\C$-linear exact equivalence preserves the Euler pairing and induces an isomorphism of numerical Grothendieck groups.  It therefore preserves the rank of the antisymmetrized Euler form and the dimension of the numerical Grothendieck group.  The equality of ranks gives equality of categorical genera and the first basket condition by Theorem~\ref{thm:cyclic-rank}.  The equality of dimensions gives the second basket condition by Proposition~\ref{prop:numerical-inertia} and Lemma~\ref{lem:untwisted-rank}.
\end{proof}

The general theorem recovers Proposition~\ref{prop:third-rank}: each $\frac13(1,1)$ point contributes $3-\gcd(3,2)=2$, so $r$ such points give a rank of $2+2r$ for the skew Euler form.  The two proofs encode the same additional contributions to the rank differently---through exceptional curves and their cross-pairings in the special McKay decomposition, or through twisted-sector eigenlines on which the Serre operator acts nontrivially.  The latter description extends the geometric calculation to arbitrary cyclic baskets.

\begin{example}[Johnson--Koll\'ar surfaces]\label{ex:johnson-kollar-categorical}
For $k\ge1$, consider a general Johnson--Koll\'ar surface
\[
X_{8k+4}\subset\Pj(2,2k+1,2k+1,4k+1).
\]
Gugiatti--Rota construct a full exceptional collection of $10+12k$ objects on its canonical stack $\cX$ \cite[Theorem 5.2]{GR}.  The basket is
\[
\frac1{4k+1}(1,1),\qquad
4\times\frac1{2k+1}(1,k)
\]
\cite[Section 4.1]{GR}.  At the first point, the width is $w=\gcd(4k+1,2)=1$ and the index is $\ell=4k+1$, giving a local genus contribution of $2k$.  At each remaining point, the greatest common divisor is $\gcd(2k+1,k+1)=1$, so the width is $w=1$ and the index is $\ell=2k+1$.  The local genus contribution is $k$.  Theorem~\ref{thm:cyclic-rank} therefore gives
\[
\boxed{g_{\mathrm{cat}}(\cX)=1+2k+4k=6k+1,}
\qquad
\boxed{\rk\chi^-_{\cX}=12k+2.}
\]
The full exceptional collection gives $\dim K_0^{\mathrm{num}}(\cX)_{\Q}=10+12k$.  The twisted summands have total dimension $4k+4(2k)=12k$. Thus the untwisted part has dimension $10$, so the Picard number $\rho(X)$ is $8$.  Since every width is $1$, no nonidentity twisted sector is fixed by the Serre operator.  On the untwisted part, $1-S_{\cX}$ has rank $2$, so its total kernel has dimension $8$, as also follows from $(10+12k)-(12k+2)=8$.

For $k=1$, the strong full collection on $X_{12}\subset\Pj(2,3,3,5)$ gives a direct check using the Euler matrix.  Exact row reduction of our transcription of the Euler matrix from \cite[Remark 5.3, Lemma 5.9, and Tables 1 and 4]{GR} gives $\rk\chi^-=14$ and $g_{\mathrm{cat}}=7=1+2+4\cdot1$.  Gugiatti--Rota's general adjoint formula for finite small subgroups of $\mathrm{GL}(2,\C)$ \cite[Theorem 3.1]{GR} also suggests that the special McKay calculation of Proposition~\ref{prop:third-rank} could in principle be extended beyond the $\frac13(1,1)$ case.
\end{example}

\begin{remark}[Smooth-point blow-ups]\label{rem:smooth-blowup}
A smooth-point blow-up leaves the basket unchanged.  Hence Theorem~\ref{thm:cyclic-rank} gives $g_{\mathrm{cat}}(\widetilde{\cX})=g_{\mathrm{cat}}(\cX)$ whenever the blown-up surface remains log del Pezzo.
\end{remark}

\section{Age and lattice points}\label{sec:age}

The preceding section expresses the local correction through the Serre operator.  Here we identify half of that correction with both the number of twisted sectors of age less than one and the number of interior lattice points in the quotient cone.  These counts yield the age interpretation of categorical genus and the lattice-point formula used in the toric comparison.

At a point of type $\frac1n(1,q)$, a generator acts on tangent vectors by $(z_1,z_2)\mapsto(\zeta z_1,\zeta^qz_2)$, where $n$ and $q$ are coprime.  The age of $\zeta^i$ is $i/n+\{iq/n\}$, where braces denote the fractional part.

\begin{lemma}[Three equal local counts]\label{lem:age-count}
For a cyclic quotient singularity $\frac1n(1,q)$, set $w=\gcd(n,q+1)$ and $\ell=n/w$.  Let $v_1,v_2$ be the primitive generators of its cone in the quotient lattice $N$, and put $\Delta=\operatorname{conv}\{0,v_1,v_2\}$.  Let $W$ be its local twisted summand.  The following three numbers agree:
\[
\frac12\rk\bigl((1-S)|_W\bigr)
=\#\{i:0<\operatorname{age}(i)<1\}
=\#(\Delta^\circ\cap N)
=\frac{n-w}{2}=\frac{w(\ell-1)}2.
\]
\end{lemma}

\begin{proof}
The sectors are indexed by $i=1,\dots,n-1$, with
\[
\operatorname{age}(i)
=\frac{i}{n}+\left\{\frac{iq}{n}\right\}.
\]
Since $q$ is coprime to $n$, both summands lie strictly between zero and one.  Their sum is an integer precisely when $n$ divides $i(q+1)$, and that integer can only be one.  Thus the age-one condition is characterized by the following equivalence:
\[
\operatorname{age}(i)=1
\quad\Longleftrightarrow\quad
n\mid i(q+1).
\]
The divisibility count in Lemma~\ref{lem:local-cyclic-rank} gives exactly $w-1$ such indices.  A nonzero self-paired index must be $i=n/2$, with $n$ even. Since $q$ and $n$ are coprime, $q$ is odd, and the age of this sector is one, as the formula gives $\operatorname{age}(n/2)=\frac12+\{q/2\}=1$.  For every remaining index, the sectors $i$ and $n-i$ are distinct and satisfy
\[
\operatorname{age}(i)+\operatorname{age}(n-i)=2.
\]
Indeed, $iq/n$ is nonintegral, so the fractional parts satisfy $\{(n-i)q/n\}=\{-iq/n\}=1-\{iq/n\}$. The first-coordinate terms also sum to one.
The other $n-w$ sectors consequently occur in pairs with one age below one and the other above one.  Lemma~\ref{lem:local-cyclic-rank} identifies half their number with half the local rank of the skew Euler form.

To see the lattice interpretation directly, identify $v_1,v_2$ with the standard basis and $N$ with $\Z^2+\Z\frac1n(1,q)$.  The nonzero cosets modulo $\Z^2$ have representatives
\[
p_i=\left(\frac in,\left\{\frac{iq}{n}\right\}\right),\qquad 1\le i<n.
\]
Both coordinates are positive, and $p_i\in\Delta^\circ$ precisely when their sum, $\operatorname{age}(i)$, is less than one.  Conversely every interior lattice point has this form.  This proves the three-way equality. Since $n$ equals $w\ell$, the final expression follows.  Equivalently, Pick's theorem gives the same count: the triangle has area $n/2$, while its three edges have lattice lengths $1,1,w$, so its total lattice boundary length, equivalently the number of boundary lattice points, is $w+2$.
\end{proof}

Figure~\ref{fig:local-triangle} illustrates the local equalities in an integral basis of $N$.
\begin{figure}[htbp]
\centering
\begin{tikzpicture}[x=1.1cm,y=0.65cm,font=\small]
  \draw[step=1,gray!20,very thin] (-1.4,-0.4) grid (1.4,5.4);
  \filldraw[fill=blue!9,draw=blue!65!black,thick] (0,0)--(1,0)--(-1,5)--cycle;
  \fill (0,0) circle (2pt) node[below left] {$0$};
  \fill (1,0) circle (1.8pt) node[below right] {$v_1=(1,0)$};
  \fill (-1,5) circle (1.8pt) node[above] {$v_2=(-1,5)$};
  \fill[blue!70!black] (0,1) circle (2.5pt);
  \fill[blue!70!black] (0,2) circle (2.5pt);
  \node[anchor=west,align=left] at (2,3.1)
    {$\frac15(1,1)$: $w=1$, $\ell=5$\\[5pt]
     $(0,1)\longleftrightarrow\operatorname{age}=\frac25$\\[4pt]
     $(0,2)\longleftrightarrow\operatorname{age}=\frac45$\\[5pt]
     local skew Euler form: rank $4$\\
     local genus $=2$};
\end{tikzpicture}
\caption{The cone triangle for $\frac15(1,1)$ in the lattice $N=\Z^2$.  Its two interior points correspond to the two age-$<1$ sectors.  Each count equals the local genus contribution, not the total categorical genus.}
\label{fig:local-triangle}
\end{figure}

\begin{definition}[Age and twisted-sector spaces]\label{def:age-twisted-sectors}
In the terminology of Ito--Reid, a stabilizer element is junior if its age is exactly one \cite[Section 1.2]{ItoReid}.  This age-one condition is distinct from the strict inequality used to define the twisted-sector space below.  For the canonical stack $\cX$, set
\[
H^{\operatorname{age}<1}_{\cX}
:=\bigoplus_{\substack{\gamma\subset I\cX\\0<\operatorname{age}(\gamma)<1}}
\C\,\1_\gamma,
\]
where the sum runs over the nonidentity connected components $\gamma$ of the inertia stack whose age is strictly less than one.  The vector $\1_\gamma$ is the basis vector corresponding to the component $\gamma$.
\end{definition}

\begin{theorem}[Age interpretation]\label{thm:age-interpretation}
Let $X$ be a log del Pezzo surface with cyclic quotient singularities, and let $\cX$ be its canonical stack.  Its categorical genus satisfies
\begin{equation}\label{eq:age-genus}
\boxed{
g_{\mathrm{cat}}(\cX)=1+\dim H^{\operatorname{age}<1}_{\cX}.}
\end{equation}
\end{theorem}

\begin{proof}
Lemma~\ref{lem:age-count} identifies the contribution of each singular point to $\dim H^{\operatorname{age}<1}_{\cX}$ with half its local contribution to the rank of the skew Euler form.  Summing over the basket and applying Theorem~\ref{thm:cyclic-rank} proves the claimed identity.
\end{proof}

\section{Residual singularities and mirror genus}\label{sec:mirror}

The age count records the twisted sectors of the canonical stack at hand.  Under general $\Q$-Gorenstein deformation, smoothable sectors can disappear, so a mirror comparison must distinguish the original stack from the canonical stack of a general fiber.  We now express the residual local contributions as lattice points in residual cones of a toric degeneration.

\begin{definition}[Fano polygons]
A Fano polygon is a lattice polygon with primitive vertices and the origin strictly inside it.  The toric surface $X_P$ associated with a Fano polygon $P$ is defined by the fan over its edges, not by its normal fan.
\end{definition}

\begin{definition}[Maximally mutable Laurent polynomials and their parameter spaces]\label{def:mmlp-parameter-space}
Let $P\subset N_{\mathbb R}$ be a Fano polygon and let
\[
f=\sum_{v\in P\cap N}a_vx^v
\]
have Newton polygon $P$.  Following \cite[Definition 4]{ACC+}, the polynomial $f$ is \emph{maximally mutable} if $a_0=0$ and, for every finite sequence of polygon mutations $P=P_0\to P_1\to\cdots\to P_n$, there are Laurent polynomials $f_i$ with $f_0=f$ and $\operatorname{Newt}(f_i)=P_i$ such that $f_i$ mutates to $f_{i+1}$ by the algebraic mutation prescribed by each step.  The vector space of maximally mutable Laurent polynomials with Newton polygon $P$ is denoted by $L_P$.

A maximally mutable Laurent polynomial is \emph{normalized} if every vertex coefficient is one \cite[Definition 2.2]{CKPT}.  Let $E$ be an edge of height $r_E$ and width
\[
w_E=m_Er_E+w_{0,E},\qquad 0\leq w_{0,E}<r_E.
\]
The polynomial $f$ has \emph{$T$-binomial edge coefficients} along $E$ if its successive coefficients on $E$ are the coefficients, in increasing powers of $T$, of
\[
\begin{cases}
(1+T)^{m_Er_E}, & w_{0,E}=0,\\
(1+T)^{m_Er_E}(1+T^{w_{0,E}}), & w_{0,E}>0.
\end{cases}
\]
The affine parameter space $L_P^T\subset L_P$ consists of the maximally mutable Laurent polynomials with these prescribed $T$-binomial coefficients along every edge.  The edge conditions force all vertex coefficients to be one, so every member of $L_P^T$ is normalized.  Thus vertex normalization and the prescribed edge coefficients are additional conditions defining $L_P^T$, not conditions imposed on every member of $L_P$.
\end{definition}

Whenever a member of one of these families is called general, this means general in the indicated coefficient space.

\begin{definition}[$\Q$-Gorenstein deformations]
A $\Q$-Gorenstein deformation is a deformation that locally comes from an equivariant deformation of the index-one cover, the cyclic cover on which the canonical class becomes Cartier.  This formulation includes the required base-change compatibility for reflexive powers of the relative canonical sheaf \cite[Basic Concepts]{ACC+}.
\end{definition}

\begin{definition}[Toric degenerations and local rigidity]
A toric $\Q$-Gorenstein degeneration of a del Pezzo surface is a $\Q$-Gorenstein degeneration with reduced fibers whose special fiber is a normal toric del Pezzo surface \cite[Definition 1]{ACC+}.  A surface is locally $\Q$-Gorenstein rigid if its singularities admit no nontrivial local $\Q$-Gorenstein deformations.
\end{definition}

\begin{definition}[Width and local Gorenstein index]
Consider a cyclic quotient $\frac1n(1,q)$.  Its cone has lattice width $w$ and lattice height $\ell$, equal to the local Gorenstein index (the height denoted $r$ in \cite[Basic Concepts]{ACC+}), determined by
\[
w=\gcd(n,q+1),\qquad n=w\ell.
\]
\end{definition}
Putting $a=(q+1)/w$ gives the useful normal form
\begin{equation}\label{eq:cyclic-normal-form}
\frac1{w\ell}(1,wa-1),\qquad \gcd(a,\ell)=1.
\end{equation}
The index-one cover is the Du Val hypersurface $xy=z^w$, with residual $\mu_\ell$-action of weights $(1,wa-1,a)$ on $(x,y,z)$.  By the index-one-cover definition, its equivariant deformations give the local $\Q$-Gorenstein deformations \cite[Section 2.1]{HP}.  On the cover, the first-order deformation space is $T^1\simeq\C[z]/(z^{w-1})$.  A monomial $z^i$ in an equivariant perturbation of $z^w$ must have the same character as $z^w$, so $ai\equiv aw\pmod\ell$.  Since $a$ is coprime to $\ell$, for $\ell>1$ the invariant classes are exactly $z^{w-k\ell}$ for $1\le k\le m$.  Thus a miniversal local $\Q$-Gorenstein deformation is given by
\[
xy=z^{w_0}F(z^\ell),\qquad
F(Z)=Z^m+t_1Z^{m-1}+\cdots+t_m.
\]
For $\ell=1$, the singularity is Du Val and the usual $A_{w-1}$ miniversal deformation has smooth general fiber.  When $w_0=0$ and $\ell>1$, the displayed family is the versal $T_m$ deformation of Hacking--Prokhorov \cite[Proposition 2.2]{HP}, and Koll\'ar--Shepherd-Barron classify the cyclic quotient $T$-singularities \cite[Proposition 3.10]{KSB}.

\begin{definition}[Smoothable and residual cyclic quotient singularities and residual baskets]\label{def:residual-basket}
The decomposition into smoothable and residual parts is described by Euclidean division:
\[
w=m\ell+w_0,
\qquad 0\le w_0<\ell.
\]
The case $w_0=0$ gives a $T$-singularity, which is $\Q$-Gorenstein smoothable.  The case $m=0$, equivalently $0<w<\ell$, gives a residual singularity, which is $\Q$-Gorenstein rigid.  In general, the residual type keeps the same $a$ in \eqref{eq:cyclic-normal-form} and replaces $w$ by $w_0$:
\[
\begin{cases}
\frac1{w_0\ell}(1,w_0a-1), & w_0>0,\\
\text{no residual point}, & w_0=0.
\end{cases}
\]
For a normal surface $X$ with isolated cyclic quotient singularities, the \emph{residual basket} $\mathcal B_{\mathrm{res}}(X)$ is the multiset, counted with multiplicity, of the nonempty residual cyclic quotient singularities obtained by applying this width decomposition at every singular point of $X$.  A point with residual width $w_0=0$ contributes no residual singularity.  The construction depends only on the local cyclic quotient types of $X$ and not on a choice of toric degeneration.  Under the globalization hypotheses of Lemma~\ref{lem:generic-residual-basket}, this multiset is the singularity basket of a sufficiently general $\Q$-Gorenstein deformation of $X$.  The residual basket in the singularity content of a Fano polygon is defined from its residual cones.  These two constructions agree under the toric degeneration hypotheses used below, but they are not definitions of the same object.
\end{definition}
Thus $m$ counts the primitive smoothable $T$-cones and $m\ell$ is the width occupied by the $T$-part.  The dense open condition that $F$ have distinct nonzero roots removes this part and leaves only the residual quotient displayed above.  A surface is locally $\Q$-Gorenstein rigid precisely when all its singularities are residual.  For example, $\frac13(1,1)$ and $\frac16(1,1)$ are residual, with $(w,\ell)=(1,3)$ and $(2,3)$, whereas $\frac14(1,1)$ is a $T$-singularity with $(w,\ell)=(2,2)$.

\begin{definition}[Singularity content]
The \emph{singularity content} of a Fano polygon is $(\sum m,\mathcal B)$, where the sum counts primitive $T$-cones and $\mathcal B$ is the multiset of its residual singularities.
\end{definition}

The existence of a toric $\Q$-Gorenstein degeneration and local $\Q$-Gorenstein rigidity are independent conditions.  The surface $\Pj(1,1,4)$ is toric and hence admits such a degeneration, but is not locally $\Q$-Gorenstein rigid.  Conversely, the complete intersection $X_{6,6}\subset\Pj(2,2,3,3,3)$ has four residual $\frac13(1,1)$ points and degree $1/3$. It is the family $X_{4,1/3}$ in \cite[Section 7.6, entry 25]{OP}.  It is one of the three families without a toric $\Q$-Gorenstein degeneration: its anticanonical system is empty, whereas a toric surface has an effective anticanonical boundary, and $h^0(-K)$ is invariant in $\Q$-Gorenstein families \cite[Section 7]{CH}.

To apply Tveiten's genus formula, we express the residual singularities of a Fano polygon in terms of lattice points.

Figure~\ref{fig:mixed-cone-schematic} illustrates the separation of smoothable and residual contributions.

\begin{figure}[htbp]
\centering
\begin{tikzpicture}[x=1cm,y=1cm,font=\small]
  \fill[orange!22] (3.5,0)--(0,3)--(7.5,3)--cycle;
  \fill[blue!9] (3.5,0)--(7.5,3)--(10,3)--cycle;
  \foreach \position in {2.5,5,7.5}
    \draw[gray!70,dashed] (3.5,0)--(\position,3);
  \draw[thick] (3.5,0)--(0,3)--(10,3)--cycle;
  \fill (3.5,0) circle (2.2pt) node[below] {$0$};
  \fill (0,3) circle (1.8pt) node[left] {$u$};
  \fill (10,3) circle (1.8pt) node[right] {$v$};
  \foreach \position/\width in {1.25/\ell,3.75/\cdots,6.25/\ell,8.75/w_0}
    \node[above] at (\position,3.05) {$\width$};
  \draw[thin] (0,3.55)--(0,3.95);
  \draw[thin] (10,3.55)--(10,3.95);
  \draw[<->] (0,3.8)--(10,3.8) node[midway,above] {$E:\quad w=m\ell+w_0$};
  \node[fill=orange!22,inner sep=2pt] at (3.5,2.45) {$m$ primitive $T$-triangles};
  \draw[gray!70,dotted] (3.5,0)--(10.9,0);
  \draw[gray!70,dotted] (10,3)--(10.9,3);
  \draw[<->] (10.75,0)--(10.75,3) node[midway,right] {$\ell$};
  \node[text=orange!65!black] at (2.7,-0.75) {$T$-part: $m\binom{\ell}{2}$};
  \node[text=blue!70!black] at (8.3,-0.75) {residue: $\dfrac{w_0(\ell-1)}2$};
  \node at (5.2,-1.65)
    {$\displaystyle\frac{w(\ell-1)}2=m\binom{\ell}{2}+\frac{w_0(\ell-1)}2$};
\end{tikzpicture}
\caption{Schematic decomposition of a mixed cone of width $w=m\ell+w_0$ and height $\ell$.  The orange part consists of $m$ primitive $T$-cones and contributes $m\binom{\ell}{2}$ to the categorical genus of the toric stack. The blue residual cone contributes $w_0(\ell-1)/2$ and survives a general $\Q$-Gorenstein deformation.  The repeated pieces are schematic, not a prescribed number of lattice triangles.}
\label{fig:mixed-cone-schematic}
\end{figure}

\begin{definition}[Residual lattice points and mutable genus]\label{def:residual-lattice-points}
Choose a crepant subdivision of the spanning fan of $P$ into primitive $T$-cones and residual cones $C_\alpha$: the inserted primitive ray generators lie on the original outer edge, so the canonical support function is unchanged and the induced toric morphism is crepant.  If $u_\alpha,v_\alpha$ are the primitive generators of $C_\alpha$, set
\[
R_{C_\alpha}
=\bigl(\operatorname{conv}\{0,u_\alpha,v_\alpha\}\cap N\bigr)
\setminus\{0,u_\alpha,v_\alpha\},
\qquad R=\bigcup_\alpha R_{C_\alpha}.
\]
The set $R$ is a \emph{residual lattice-point set} for the chosen subdivision.  Following Tveiten's terminology \cite[Theorem 3.13]{Tveiten} and the residual-point formulation of \cite[Definitions 3.5--3.6]{CKPT}, the \emph{mutable genus} of $P$ is
\[
\boxed{g_{\mathrm{mut}}(P):=1+\#(R\cap P^\circ).}
\]
The set $R$ may depend on the crepant subdivision, whereas the cardinality of $R\cap P^\circ$, and hence $g_{\mathrm{mut}}(P)$, does not.
\end{definition}
A mixed cone is subdivided before its residual lattice points are taken, so only its residual subcone contributes to $R$.  The set $R$ may contain points of $\partial P$, but primitive radial edges have no interior lattice points.  Consequently, $R\cap P^\circ$ consists precisely of the interiors of the residual triangles.  Each residual triangle of width $w_0$ and height $\ell$ contributes $w_0(\ell-1)/2$ points by Lemma~\ref{lem:age-count}.  This proves that the interior cardinality in Definition~\ref{def:residual-lattice-points} is independent of the chosen subdivision.  Example~\ref{ex:mixed-cone-delta} exhibits two different residual sets with the same interior cardinality.  We use Tveiten's theorem \cite[Theorem 3.13]{Tveiten} in the following formulation \cite[Theorem 3.13]{CKPT}.

\begin{citedtheorem}[Tveiten's maximally mutable Laurent polynomial genus formula]
Let $f$ be a maximally mutable Laurent polynomial in two variables with Newton polygon $P$, and let $R$ be a choice of residual points as above.  Resolve the toric surface defined by the normal fan of $P$ and the basepoints of the pencil defined by $1$ and $f$.  The smooth projective model of a general member has genus
\[
g\bigl(\widetilde{\{f=\eta\}}\bigr)=g_{\mathrm{mut}}(P)=1+\#(R\cap P^\circ).
\]
\end{citedtheorem}

The cited theorem does not require rigidity of the maximally mutable Laurent polynomial. The word ``general'' refers to the member of its pencil.  We apply the cited theorem to general polynomials in the relevant maximally mutable Laurent polynomial families with the specified Newton polygon.

The polygon can be chosen without assuming a classification conjecture.  By hypothesis, choose a $\Q$-Gorenstein degeneration of $X$ to a toric surface $X_{P_0}$ and take its fan polygon $P_0$.  One may also replace $P_0$ by a mutation-equivalent polygon: mutation preserves the residual basket \cite[Theorem 3 and its proof]{ACC+}.  We do not require the conjectural injectivity of the assignment from mutation classes to $\Q$-Gorenstein deformation classes.

The local deformation families explain why the basket can vary on nongeneral fibers.  For $\ell>1$, the index-one-cover family above is $xy=z^{w_0}F(z^\ell)$, with $F$ monic of degree $m$.  A root at zero of multiplicity $b$ retains a point of width $w_0+b\ell$ with the same residue, while a multiple nonzero root gives a smoothable Du Val point because the residual group acts freely away from $z=0$.  For $\ell=1$, the singularity is Du Val and its general deformation is smooth.  Thus a nongeneral fiber can contain mixed singularities, not merely a disjoint union of residual and $T$-singularities.  If $X$ is locally $\Q$-Gorenstein rigid, no smoothable singularity can occur in its basket, which is therefore the residual basket of $P_0$.  The global-to-local deformation result of \cite[Lemma 6]{ACC+} and mutation invariance \cite[Theorem 3 and its proof]{ACC+} give the same residual basket for $P$.

\begin{theorem}[The locally rigid mirror-genus comparison]\label{cor:genus-match}

Let $X$ be a locally $\Q$-Gorenstein rigid log del Pezzo surface with cyclic quotient singularities.  Assume that $X$ admits a toric $\Q$-Gorenstein degeneration, let $\cX$ be its canonical stack, and choose a $\Q$-Gorenstein degeneration to $X_{P_0}$.  Let $P$ be $P_0$ or a mutation-equivalent Fano polygon.  For a general maximally mutable Laurent polynomial $f$ with Newton polygon $P$ and a general value $\eta$, the smooth projective model of the fiber satisfies
\[
\boxed{
g_{\mathrm{cat}}(\cX)
=g\bigl(\widetilde{\{f=\eta\}}\bigr).}
\]
\end{theorem}

\begin{proof}
The preceding local deformation discussion identifies the basket of $X$ with the residual basket of $P$.  Lemma~\ref{lem:age-count} gives $\#(R\cap P^\circ)=\dim H^{\operatorname{age}<1}_{\cX}$, while Tveiten's formula \cite[Theorem 3.13]{CKPT} gives genus $1+\#(R\cap P^\circ)$ for the general fiber.  The age interpretation theorem identifies this with $g_{\mathrm{cat}}(\cX)$.
\end{proof}

\begin{remark}\label{rem:local-rigidity-needed}
For $\Pj(1,1,4)$, Proposition~\ref{prop:weighted-rank} gives $g_{\mathrm{cat}}=2$, but its $\frac14(1,1)$ point is $\Q$-Gorenstein smoothable and contributes no residual point.  For a general maximally mutable Laurent polynomial $f$ with this polygon and a general value $\eta$, the fiber $f=\eta$ has genus one.  Thus local $\Q$-Gorenstein rigidity cannot simply be dropped from the comparison with the categorical genus of the given stack.  The categorical and age formulas remain valid.  Section~\ref{sec:qg-generic} instead compares the genus of a general maximally mutable Laurent-polynomial fiber with the residual categorical genus of a sufficiently general $\Q$-Gorenstein deformation.
\end{remark}

For the Corti--Heuberger families discussed in the introduction, Theorem~\ref{cor:genus-match} identifies the categorical value $k+1$ with mirror genus whenever the family admits a toric $\Q$-Gorenstein degeneration.  Remark~\ref{rem:smooth-blowup} also shows that this value is constant along smooth-locus blow-ups in their cascades.

\begin{example}[The $X_{10}$ mirror polygon]\label{ex:x10-mirror-polygon}
For $X_{10}$, the case $r=1$ gives $g_{\mathrm{cat}}(\cX_{10})=2$.  Its mirror mutation class contains
\[
P=\operatorname{conv}\{(-3,-5),(-3,5),(7,5)\},
\]
with singularity content $(10,\{\frac13(1,1)\})$ \cite[Table 4 and Theorem 9]{KNP}.  Here $10$ counts primitive $T$-cones and the residual basket consists of one $\frac13(1,1)$ point.  For a general maximally mutable Laurent polynomial $f$ with Newton polygon $P$ and a general value $\eta$, the fiber $f=\eta$ has genus two, matching $g_{\mathrm{cat}}(\cX_{10})=2$.  Proposition~\ref{prop:x10-explicit-model} derives the explicit hyperelliptic model from the Laurent polynomial in \cite[Section 7.3, entry 11]{OP}.
\end{example}

Figure~\ref{fig:x10-subdivision} separates the residual point from the smoothable content. Example~\ref{ex:x10-toric-drop} computes the resulting drop of $34$.
\begin{figure}[htbp]
\centering
\begin{tikzpicture}[x=0.61cm,y=0.61cm,font=\small]
  \fill[orange!16] (-3,-5)--(7,5)--(-3,5)--cycle;
  \fill[blue!22] (0,0)--(-3,-4)--(-3,-5)--cycle;
  \foreach \horizontal/\vertical in {-3/-5,-1/-3,1/-1,3/1,5/3,7/5,2/5,-3/5,-3/2,-3/-1,-3/-4}
    \draw[gray!70,dashed,thin] (0,0)--(\horizontal,\vertical);
  \draw[blue!70!black,thick] (-3,-4)--(0,0)--(-3,-5);
  \draw[thick] (-3,-5)--(7,5)--(-3,5)--cycle;
  \foreach \horizontal/\lowest in {-2/-3,-1/-2,0/-1,1/0,2/1,3/2,4/3,5/4}
    \foreach \vertical in {\lowest,...,4}
      \fill[orange!65!black] (\horizontal,\vertical) circle (1.6pt);
  \fill[blue!75!black] (-2,-3) circle (2.7pt);
  \fill[white] (0,0) circle (3.5pt);
  \fill (0,0) circle (2.4pt) node[above left] {$0$};
  \node[above left] at (-3,5) {$(-3,5)$};
  \node[above right] at (7,5) {$(7,5)$};
  \node[below] at (-3,-5) {$(-3,-5)$};
  \node[above,align=center] at (2,5.6) {$2$ $T$-cones, $\ell=5$};
  \node[anchor=east,align=right] at (-3.35,1.3) {$3$ $T$-cones\\$\ell=3$};
  \node[anchor=west,align=left] at (4,0.1) {$5$ $T$-cones\\$\ell=2$};
  \node[anchor=east,align=right,text=blue!70!black] at (-3.4,-3.9) {residual cone\\$w_0=1$, $\ell=3$};
  \draw[blue!70!black,->] (-3.4,-3.9)--(-2.55,-3.85);
  \node[anchor=west,align=left] at (1,-3.5)
    {$34$ orange points: $T$-content\\
     $1$ blue point $+$ origin: genus $2$};
\end{tikzpicture}
\caption{A crepant subdivision of the $X_{10}$ mirror polygon into ten primitive $T$-cones (orange) and one residual cone (blue).  The $36$ interior points comprise $34$ points in the $T$-triangles, one residual point, and the origin.  Thus the toric canonical stack has categorical genus $36$, while the residual and maximally mutable Laurent polynomial genera are $2$.}
\label{fig:x10-subdivision}
\end{figure}

The polygon has substantial smoothable $T$-content.  What happens to categorical genus when that content is removed?  The next section answers this question under general $\Q$-Gorenstein deformation and identifies the surviving value with mirror genus.

\section{\texorpdfstring{$\Q$-Gorenstein}{Q-Gorenstein} deformation and the mirror-side genus drop}\label{sec:qg-generic}

Section~\ref{sec:mirror} proves the mirror comparison when local $\Q$-Gorenstein rigidity ensures that the original basket is already residual.  We now treat arbitrary cyclic quotient baskets by passing to the canonical stack of a sufficiently general $\Q$-Gorenstein deformation.  The residual formula and exact $T$-content drop are intrinsic and require no toric degeneration.  A toric degeneration is needed only for the mirror identification.

\begin{lemma}[The generic residual basket]\label{lem:generic-residual-basket}
Let $X$ be a log del Pezzo surface with cyclic quotient singularities, and let $\mathfrak X\to(B,0)$ be its analytic miniversal $\Q$-Gorenstein deformation germ with $X_0=X$.  The germ $B$ is smooth and irreducible, and it has a dense open locus $B^\circ$ on which the fibers $X_t$ have exactly the residual basket of $X$.  After shrinking $B$, these fibers remain log del Pezzo.
\end{lemma}

\begin{proof}
We use $B$ for the analytic miniversal germ, whose completion is the formal miniversal base.  The formal assertions below do not by themselves assert the existence of a global algebraic family.  We use \cite[Lemma 6]{ACC+} for its global unobstructedness and formal smoothness assertions.  The local miniversal bases and their parameter counts are computed above from the index-one covers.  It follows that the completed global base is regular and the completed global-to-local map is formally smooth.  Regularity and smoothness can be checked on completions for these analytic germs, so after shrinking, $B$ is smooth and irreducible and its map to the product of local bases is smooth and dominant.

For $\ell>1$, the conditions that $F(0)\ne0$ and that $F$ have distinct roots form a dense open set in each local miniversal base.  Its fibers retain exactly the residual quotient at the origin when $w_0>0$ and are smooth there when $w_0=0$.  For $\ell=1$, the dense open set of the usual $A_{w-1}$ deformation consists of smooth fibers.  Smoothness and dominance of the global-to-local map make the inverse image of the product of these local open sets nonempty and open in $B$, hence dense.  To exclude singularities elsewhere, choose disjoint neighborhoods of the original singular points.  The family is flat and the central fiber is smooth outside these neighborhoods, so the morphism is smooth along that part of the central fiber.  Properness then lets us shrink $B$ so that no fiber acquires a singular point outside the chosen neighborhoods.  Finally, relative $\Q$-Cartierness of $K_{\mathfrak X/B}$ and openness of relative ampleness of $-K_{\mathfrak X/B}$ ensure that the nearby fibers remain log del Pezzo.
\end{proof}

\begin{definition}[Categorical genus under general $\Q$-Gorenstein deformation]\label{def:qg-generic-genus}
Choose the analytic miniversal family $\mathfrak X\to(B,0)$ as in Lemma~\ref{lem:generic-residual-basket}.  For $t\in B^\circ$, let $\mathcal X_t$ denote the canonical stack of the fiber $X_t$.  Define
\[
g_{\mathrm{cat}}^{\mathrm{qG}}(X)
:=g_{\mathrm{cat}}(\mathcal X_t).
\]
The superscript $\mathrm{qG}$ records the $\Q$-Gorenstein deformation used to compute the invariant.
\end{definition}

By Lemma~\ref{lem:generic-residual-basket}, every fiber in the dense open locus has the same residual basket, determined by the local types of $X$.  The categorical genus formula therefore gives the same value for every such fiber and every miniversal realization.  The residual categorical genus is thus an invariant of the deformation germ at $X$, not of a component of a global moduli space.  Its value is given by the following formula.

\begin{proposition}[Residual formula]\label{prop:qg-residual-genus}
Let $X$ be a cyclic quotient log del Pezzo surface, and use the local data $(w_j,\ell_j,m_j,w_{0,j})$ of Section~\ref{sec:mirror} at its singular points.  The categorical genus under general $\Q$-Gorenstein deformation is
\begin{equation}\label{eq:qg-generic-genus}
\boxed{g_{\mathrm{cat}}^{\mathrm{qG}}(X)
=1+\frac12\sum_j w_{0,j}(\ell_j-1).}
\end{equation}
\end{proposition}

\begin{proof}
Use the normal form \eqref{eq:cyclic-normal-form}, with $a=(q+1)/w$, where $a$ and $\ell$ are coprime.  Lemma~\ref{lem:generic-residual-basket} removes the $m\ell$-part of the width, leaving the residue with the same $a$:
\[
\frac1{w_0\ell}(1,w_0a-1)
\quad\text{if }w_0>0,
\]
and there is no residual point if $w_0=0$.  This is a reduced cyclic quotient: since $w_0a-1$ is congruent to $wa-1=q$ modulo $\ell$, it is coprime to $\ell$, and the congruence $w_0a-1\equiv-1\pmod{w_0}$ shows that it is also coprime to $w_0$.  The greatest common divisor determining its width is
\[
\gcd(w_0\ell,w_0a)=w_0.
\]
Theorem~\ref{thm:cyclic-rank} therefore gives the local genus contribution
$(w_0\ell-w_0)/2=w_0(\ell-1)/2$.  Summing over the generic residual basket proves \eqref{eq:qg-generic-genus}.
\end{proof}

\begin{corollary}[The $T$-content drop]\label{cor:t-content-drop}
For a cyclic quotient log del Pezzo surface $X$, with the notation of Proposition~\ref{prop:qg-residual-genus}, one has
\begin{equation}\label{eq:t-content-drop}
\boxed{g_{\mathrm{cat}}(\cX)-g_{\mathrm{cat}}^{\mathrm{qG}}(X)
=\sum_j m_j\binom{\ell_j}{2}.}
\end{equation}
Equivalently, for $t\in B^\circ$ in Lemma~\ref{lem:generic-residual-basket}, the ranks of the skew Euler forms of $\cX$ and the canonical stack $\mathcal X_t$ of $X_t$ satisfy
\[
\boxed{\rk_{\Q}\chi^-_{\cX}-\rk_{\Q}\chi^-_{\mathcal X_t}
=\sum_j m_j\ell_j(\ell_j-1).}
\]
\end{corollary}

\begin{proof}
Theorem~\ref{thm:cyclic-rank} gives the local genus correction
$(n_j-w_j)/2=w_j(\ell_j-1)/2$ on $\cX$, whereas Proposition~\ref{prop:qg-residual-genus} gives $w_{0,j}(\ell_j-1)/2$ on the canonical stack of a sufficiently general fiber.  Their difference is
\[
\frac{w_j(\ell_j-1)}2-\frac{w_{0,j}(\ell_j-1)}2
=\frac{m_j\ell_j(\ell_j-1)}2
=m_j\binom{\ell_j}{2}.
\]
Summing proves \eqref{eq:t-content-drop}.  Doubling gives the corresponding formula for the difference between the ranks of the skew Euler forms.
\end{proof}

The preceding formulas are intrinsic to $X$.  When $X$ is toric, both pieces acquire a lattice interpretation.  We first count all interior lattice points of the full polygon, before removing any $T$-content.

\begin{proposition}[Toric categorical genus]\label{prop:toric-categorical-genus}
Let $P\subset N_{\mathbb R}$ be a Fano polygon in a rank-two lattice $N$, and let $X_P$ be the toric del Pezzo surface whose fan consists of the cones over the faces of $P$.  Its canonical stack satisfies
\begin{equation}\label{eq:toric-categorical-genus}
\boxed{g_{\mathrm{cat}}(\mathcal X_P)=\#(P^\circ\cap N).}
\end{equation}
\end{proposition}

\begin{proof}
Order the primitive vertices $v_1,\dots,v_r$ counterclockwise, with $v_{r+1}=v_1$, and put $\Delta_i=\operatorname{conv}\{0,v_i,v_{i+1}\}$.  The cone over the $i$th edge has determinant $n_i>0$, edge width $w_i$, and local Gorenstein index $\ell_i$, with $n_i=w_i\ell_i$.  For a singular cone of type $\frac1{n_i}(1,q_i)$, its width is $w_i=\gcd(n_i,q_i+1)$.  For a smooth cone we take $n_i=w_i=\ell_i=1$, so its local correction is zero.

Lemma~\ref{lem:age-count} gives $\#(\Delta_i^\circ\cap N)=(n_i-w_i)/2$, the local genus correction in Theorem~\ref{thm:cyclic-rank}. For a smooth cone both numbers are zero.  The triangles cover $P$ and meet along radial segments.  Such a segment contains no lattice points other than $0$ and its primitive endpoint, which lies on $\partial P$.  Thus every nonzero lattice point in $P^\circ$ belongs to exactly one triangle interior.  The origin is itself an interior lattice point of the Fano polygon and contributes the initial $1$.  The point count therefore satisfies
\[
\#(P^\circ\cap N)
=1+\sum_i\frac{n_i-w_i}{2}
=g_{\mathrm{cat}}(\mathcal X_P).
\]
\end{proof}

\begin{remark}[A toric numerical consistency check]\label{rem:toric-numerical-check}
In the setting of Proposition~\ref{prop:toric-categorical-genus}, let $r$ be the number of rays and $n_i$ the cone determinants, including $n_i=1$ for smooth cones.  The toric divisor sequence gives $\rho(X_P)=r-2$ \cite[Theorem 4.1.3 and Proposition 4.2.7]{CLS}.  Proposition~\ref{prop:numerical-inertia} and the untwisted decomposition in Lemma~\ref{lem:untwisted-rank} therefore give
\[
\dim_{\Q}K_0^{\mathrm{num}}(\mathcal X_P)_{\Q}
=2+(r-2)+\sum_{i=1}^r(n_i-1)
=\sum_{i=1}^r n_i=2\operatorname{Area}_N(P),
\]
where a fundamental parallelogram of $N$ has area one.  This uses the canonical stack, with no generic or divisorial stabilizers.  Writing $I=\#(P^\circ\cap N)$ and $B=\#(\partial P\cap N)$, Proposition~\ref{prop:toric-categorical-genus} and Pick's theorem yield
\[
\rk_{\Q}\chi^-_{\mathcal X_P}=2I,
\qquad
2\operatorname{Area}_N(P)=2I+B-2.
\]
By Lemma~\ref{lem:skew-serre}, the dimension of the kernel of $1-S_{\mathcal X_P}$ on the numerical Grothendieck group is
\[
\dim_{\Q}\ker(1-S_{\mathcal X_P})=B-2.
\]
Together with the genus formula for a general Laurent polynomial in Corollary~\ref{cor:toric-laurent-genus}, these identities provide a numerical consistency check for the comparison between the Euler and Seifert pairings \cite[Sections 4.2 and 5]{HarderThompson}.
\end{remark}

\begin{definition}[Nondegeneracy with respect to a polygon]\label{def:nondegenerate-laurent}
A Laurent polynomial $g$ with Newton polygon $P$ is \emph{nondegenerate with respect to $P$} if, for every face $F$ of $P$, including $P$ itself, the face polynomial $g_F$ and its logarithmic partial derivatives $x\partial_x g_F$ and $y\partial_y g_F$ have no common zero in $(\C^\times)^2$.
\end{definition}

The same full interior-lattice-point count gives the genus of a general Laurent-polynomial fiber, providing the first bridge from categorical genus to curve geometry.

\begin{corollary}[General Laurent-polynomial fibers]\label{cor:toric-laurent-genus}
With $P$ as in Proposition~\ref{prop:toric-categorical-genus}, let $h\in\C[N]$ be a general Laurent polynomial with Newton polygon $P$.  For general $\eta$, the smooth projective model of its fiber satisfies
\[
\boxed{g_{\mathrm{cat}}(\mathcal X_P)
=g\bigl(\widetilde{\{h=\eta\}}\bigr).}
\]
\end{corollary}

\begin{proof}
For general coefficients and $\eta$, the polynomial $h-\eta$ is nondegenerate with respect to $P$ in the sense of Definition~\ref{def:nondegenerate-laurent}.  A general member is irreducible by Bertini for the toric linear system of the two-dimensional polygon $P$.  For its smooth projective model $C$, Khovanskii's formula \cite[Section 1, Theorem 1]{KhovanskiiGenus} gives $\chi(\mathcal O_C)=1-\#(P^\circ\cap N)$, and therefore the genus of $C$ is $\#(P^\circ\cap N)$.  Proposition~\ref{prop:toric-categorical-genus} identifies this with categorical genus.
\end{proof}

Here the curve is compactified using the normal fan of its Newton polygon, distinct from the fan over the faces of $P$ defining $X_P$.  Maximal mutability changes the genus count: the smooth projective model of a general fiber of a general Laurent polynomial records all interior lattice points, whereas that of a general maximally mutable Laurent polynomial records only the residual interior points together with the origin.  It is therefore the residual categorical genus that enters the mirror comparison.

\begin{corollary}[$\Q$-Gorenstein generic mirror genus]\label{cor:qg-generic-mirror-genus}
Let $X$ be a cyclic quotient log del Pezzo surface admitting a toric $\Q$-Gorenstein degeneration to $X_{P_0}$, and let $P$ be $P_0$ or a mutation-equivalent Fano polygon.  For a general maximally mutable Laurent polynomial $f$ with Newton polygon $P$ and general $\eta$, one has
\begin{equation}\label{eq:qg-generic-mirror-genus}
\boxed{g_{\mathrm{cat}}^{\mathrm{qG}}(X)
=g_{\mathrm{mut}}(P)
=1+\#(R\cap P^\circ)
=g\bigl(\widetilde{\{f=\eta\}}\bigr),}
\end{equation}
where $R$ is the residual set of $P$.
\end{corollary}

\begin{proof}
By Proposition~\ref{prop:qg-residual-genus}, the left-hand side is $1+\frac12\sum_jw_{0,j}(\ell_j-1)$.  Lemma~\ref{lem:generic-residual-basket} and residual-basket invariance under $\Q$-Gorenstein deformation and polygon mutation \cite[Theorem 3 and its proof]{ACC+} identify the residual basket of the general deformation with that of $P$.  Lemma~\ref{lem:age-count} then identifies each nonempty residue with $w_{0,j}(\ell_j-1)/2$ interior residual points, and the maximally mutable Laurent polynomial genus formula \cite[Theorem 3.13]{CKPT} proves the remaining equality.
\end{proof}

For the parameter count, use the affine space $L_P^T$ of Definition~\ref{def:mmlp-parameter-space}.  The coefficient parametrization in \cite[Proposition 3.7]{CKPT} gives one free coefficient for each residual lattice point, while the prescribed $T$-binomial edge coefficients fix the boundary coefficients.  It follows that the dimension of $L_P^T$ is given by
\[
\dim L_P^T=\#(R\cap P^\circ).
\]
Applying the age count to the canonical stack $\cX_t$ of a sufficiently general $\Q$-Gorenstein fiber and using \eqref{eq:qg-generic-mirror-genus} therefore gives
\[
g_{\mathrm{cat}}^{\mathrm{qG}}(X)-1
=\dim H^{\operatorname{age}<1}_{\cX_t}
=\dim L_P^T.
\]
This is the numerical dimension comparison relevant to Conjecture~B, not its affine-linear period-matching assertion.

\begin{remark}[Choice of toric degeneration]\label{rem:choice-of-toric-degeneration}
Theorem~3 of \cite{ACC+} gives a surjection from polygon mutation classes to $\Q$-Gorenstein deformation classes of locally $\Q$-Gorenstein rigid surfaces admitting toric $\Q$-Gorenstein degenerations.  More generally, if $X$ admits two toric $\Q$-Gorenstein degenerations, Lemma~\ref{lem:generic-residual-basket} gives the same residual basket for their sufficiently general fibers, even when $X$ itself is not locally rigid.  Corollary~\ref{cor:qg-generic-mirror-genus} then gives the same mirror genus for both polygons.  They need not be mutation-equivalent: injectivity of the surjection is the separate assertion of Conjecture~A in \cite{ACC+} and is not used here.
\end{remark}

The numerical loss in Corollary~\ref{cor:t-content-drop} raises a geometric question: where does the genus drop occur on the mirror curve?  The categorical calculation determines its size, while Corollary~\ref{cor:qg-generic-mirror-genus} identifies the genus of the smooth projective model that remains after normalization.

Tveiten observes that when $m$ mutation factors on an edge of height $\ell$ coincide, the resulting point is not an ordinary $m\ell$-fold point: it is an $\ell$-fold point whose branches have order-$m$ tangency, with total genus defect $m\binom{\ell}{2}$ \cite[Remark 3.14, arXiv version 2]{Tveiten}.  The same remark notes that special fibers such as $f=0$ may have smaller genus when the constant coefficient vanishes.  Our comparison uses a general member $f=\eta$, for which the constant coefficient of $f-\eta$ is $-\eta$.

Lutz describes the corresponding boundary basepoints.  His mutable cycle records the points and multiplicities prescribed by mutation factors \cite[Definition 3.1]{Lutz}.  Theorem~3.3 and Corollary~3.4 of \cite{Lutz} identify the associated sections after successive boundary blow-ups, subtracting the edge height $h_E=\ell_E$ from each exceptional class.  For the normalized $T$-part, the $m_E$ successive centers lie over $x=-1$.  After the first blow-up they are infinitely near, at intersections of the strict boundary transform with the new exceptional divisor.  Lemma~3.2 of \cite{Lutz} gives the weighted local vanishing and multiplicity at least $\ell_E$ at each center.  Equality holds for general fibers under the hypotheses below.  We use this general construction, not the elliptic-fibration conclusion for $T$-polygons whose pencil is resolved in \cite[Lemma 3.6]{Lutz}.  Proposition~\ref{prop:boundary-delta} makes the geometry explicit in the normalized maximally mutable Laurent-polynomial coordinates used here and identifies its $\delta$-invariant with the independently computed categorical $\Q$-Gorenstein genus drop.

\begin{samepage}
For $m>0$ and $\ell>1$, the hypotheses of Proposition~\ref{prop:boundary-delta} give the chain
\[
\boxed{\begin{gathered}
\text{maximal mutability}\Longrightarrow\text{mutable boundary data}\\
\Longrightarrow\text{$m$ successive boundary centers of multiplicity $\ell$}\\
\Longrightarrow\text{boundary singularity}\Longrightarrow\delta=m\binom{\ell}{2}.
\end{gathered}}
\]
\end{samepage}
For $\ell=1$, the point is smooth and contributes zero.  Figure~\ref{fig:boundary-delta} illustrates the geometry under the proposition's hypotheses.

\begin{figure}[!htbp]
\centering
\begin{tikzpicture}[x=1cm,y=1cm,xscale=1,font=\normalsize]
  \useasboundingbox (-0.15,-2.05) rectangle (14.1,3.5);
  \begin{scope}
    \node[align=center] at (2.05,3.1) {(a) Mutable point and\\residual intersections};
    \draw[thick] (0,1.25)--(4.05,1.25);
    \node[above left] at (0,1.25) {$D_E$};
    \foreach \position in {2.7,3.65} {
      \draw[blue!70!black,thick] (\position-0.15,0.65)--(\position+0.15,1.85);
      \fill[blue!70!black] (\position,1.25) circle (2.5pt);
    }
    \fill[orange!65!black] (1.05,1.25) circle (3pt);
    \node[below,align=center] at (1.05,1.05) {$p_E$\\$x=-1$};
    \node[text=blue!70!black,align=center] at (3.15,2.2) {simple residual\\roots};
    \node at (2.05,-0.55) {$(1+x)^{m\ell}r_E(x)$};
    \node at (2.05,-1.15) {$r_E(-1)\ne0$};
  \end{scope}
  \begin{scope}[xshift=4.9cm,xscale=1.05]
    \node[align=center] at (2,3.1) {(b) Weighted Newton\\segment};
    \draw[->] (0.35,0.2)--(3.85,0.2) node[right] {$t$};
    \draw[->] (0.35,0.2)--(0.35,2.55) node[above] {$u$};
    \draw[orange!65!black,thick] (0.35,2.1)--(3.25,0.2);
    \fill[orange!65!black] (0.35,2.1) circle (2pt);
    \fill[orange!65!black] (3.25,0.2) circle (2pt);
    \node[above right] at (0.35,2.1) {$(0,\ell)$};
    \node[below] at (3.25,0.15) {$(m\ell,0)$};
    \node[align=center] at (2,-0.9)
      {$\operatorname{wt}(t)=1$\\$\operatorname{wt}(u)=m$};
    \node at (2,-1.65) {$F_0=t^{m\ell}Q\!\left(\dfrac{u}{t^m}\right)$};
  \end{scope}
  \begin{scope}[xshift=9.8cm]
    \node[align=center] at (2,3.1) {(c) Smooth branches\\with contact $m$};
    \draw[gray!70,->] (0.15,0.35)--(3.85,0.35) node[right] {$t$};
    \draw[gray!70,->] (2,0.05)--(2,2.15);
    \node[gray!70,anchor=west] at (2.05,1.75) {$u$};
    \foreach \coefficient in {0.3,0.6,0.9}
      \draw[orange!65!black,thick,domain=-1.4:1.4,samples=45]
        plot ({2+\x},{0.35+\coefficient*\x*\x});
    \fill[orange!65!black] (2,0.35) circle (2.5pt);
    \node[text=orange!65!black] at (2,2.4) {$\ell$ branches};
    \node at (2,-0.35) {$u=\alpha_jt^m+\cdots$};
    \node at (2,-0.9) {$I(C_i,C_j)=m$};
    \node[align=center] at (2,-1.65)
      {$\delta_E=\sum_{i<j}I(C_i,C_j)$\\[3pt]$=m\binom{\ell}{2}$};
  \end{scope}
\end{tikzpicture}
\caption{Boundary geometry under the hypotheses of Proposition~\ref{prop:boundary-delta}.  The orange point $x=-1$ carries the smoothable edge factor, while the blue residual roots meet the boundary transversely.  Panel (c) shows three representative tangent branches to illustrate $m>1$, not a fixed value of $\ell$.}
\label{fig:boundary-delta}
\end{figure}

\begin{samepage}
\begin{proposition}[$T$-content as a boundary $\delta$-invariant]\label{prop:boundary-delta}
Let $P$ be a Fano polygon and let $f$ belong to the normalized maximally mutable family of \cite[Propositions 3.7 and 3.9]{CKPT}.  For each edge $E$, write $w_E=m_E\ell_E+w_{0,E}$, where $w_{0,E}$ satisfies $0\le w_{0,E}<\ell_E$.  Up to a monomial, write its edge polynomial as $(1+x)^{m_E\ell_E}r_E(x)$ in a primitive edge coordinate.  Assume that every residual factor $r_E$ has simple roots, that $r_E(-1)\ne0$ whenever $m_E>0$, and that the general fiber is irreducible.

Let $\overline C_\eta$ be the closure of $f=\eta$ on the toric surface defined by the normal fan of $P$.  For general $\eta$ and each edge with $m_E>0$, the point $x=-1$ on its boundary divisor has $\ell_E$ smooth branches with pairwise intersection multiplicity $m_E$.  Its local invariant, with $\delta_E=0$ for $m_E=0$, is
\[
\boxed{\delta_E=m_E\binom{\ell_E}{2}.}
\]
These are the only possible singularities of $\overline C_\eta$. Their total genus defect is given by
\[
\boxed{
g_{\mathrm{cat}}(\mathcal X_P)-g\bigl(\widetilde{\overline C_\eta}\bigr)
=\sum_E\delta_E=\sum_E m_E\binom{\ell_E}{2}.}
\]
\end{proposition}
\end{samepage}

The nonvanishing condition on $r_E(-1)$ is substantive.  In the subfamily with prescribed $T$-binomial edge coefficients, $r_E(T)=1+T^{w_{0,E}}$ for $w_{0,E}>0$, and $r_E(T)=1$ for $w_{0,E}=0$.  Thus, when $m_E>0$, the condition holds for even residual width and fails for odd residual width in this subfamily.  For a general normalized maximally mutable polynomial, the residual factor need not have this prescribed form, so its value at $-1$ must be checked separately.  The boundary calculation below is asserted only under the stated edgewise hypotheses.

\begin{proof}
Fix an edge with $m=m_E>0$ and $\ell=\ell_E$, choose coordinates with this edge at height $\ell$, and write $f=\sum_h f_h(x)y^h$.  By \cite[Propositions 3.7 and 3.9]{CKPT}, the normalized maximally mutable family has factor $(1+x)^m$ and slice conditions $(1+x)^{mh}\mid f_h(x)$ for every $h>0$.  We use these conditions to derive the boundary branches, not to regard coincident mutation points as ordinary multiple points.

At the boundary point, set $t=x+1$ and $u=y^{-1}$.  After multiplying by $u^\ell$ and a unit, write the local equation as $F(t,u)=0$.  The edge is $h=\ell$, so there are no slices with $h>\ell$.  For $1\le h<\ell$, slice divisibility gives $t^{mh}\mid f_h(-1+t)$, and the top slice starts with $t^{m\ell}r_E(-1)$ up to a unit.  Slices with $h<0$ have weight greater than $m\ell$.  The $h=0$ slice makes $c_\ell$ a nonconstant affine function of $\eta$.  Thus the complete lowest-weight part for $\operatorname{wt}(t)=1$ and $\operatorname{wt}(u)=m$ is
\[
F_0(t,u)=\sum_{i=0}^{\ell}c_i t^{m(\ell-i)}u^i
=t^{m\ell}Q(u/t^m),
\]
where $c_0\ne0$ and the coefficients $c_i$ for $i<\ell$ are fixed.  In particular, for general $\eta$, $c_\ell\ne0$, so $F(0,u)=u^\ell(c_\ell+O(u))$ and the Weierstrass preparation theorem shows that the local curve has exactly $\ell$ branches counted with multiplicity.  The polynomial $Q(z)=\sum_{i=0}^{\ell}c_i z^i$ has degree $\ell$ and no zero root.  To check simplicity, set
\[
G(z)=\sum_{i<\ell}c_i z^{i-\ell},\qquad Q(z)=z^\ell\bigl(G(z)+c_\ell\bigr).
\]
A repeated root must satisfy $G'(z)=0$ and $G(z)=-c_\ell$.  Thus $-c_\ell$ must be a critical value of the nonconstant rational function $G$.  There are only finitely many such values, and $c_\ell$ varies affinely and nontrivially with $\eta$.  Hence for general $\eta$ all $\ell$ roots are distinct and nonzero.  The support of $F$ lies in the half-space of weight at least $m\ell$, and both endpoint coefficients of this compact Newton face are nonzero.  Thus the face is exactly the segment joining $(m\ell,0)$ to $(0,\ell)$, and its initial form accounts for all $\ell$ branches.

Substituting $u=t^mz$ and dividing by $t^{m\ell}$ gives $Q(z)+tH(t,z)=0$ with $H$ analytic.  The implicit function theorem at each root $\alpha_j$ gives a branch
\[
u=\alpha_jt^m+O(t^{m+1}),\qquad j=1,\ldots,\ell.
\]
\begin{samepage}
Together with the Weierstrass count, these are all the local branches.  Each is smooth, and for $j\ne k$ the difference of their graph functions has order $m$, so their intersection multiplicity is $m$.  The formula for the delta invariant of a reduced union of smooth branches therefore gives $\delta_E=m\binom\ell2$.
\end{samepage}

For each fixed $f$, the critical locus in the torus has finitely many irreducible components, and $f$ is constant on each component.  A general $\eta$ therefore avoids its finitely many critical values, so the fiber is smooth in the torus.  Simple residual roots give transverse boundary intersections, and the normalized family has nonzero coefficients at every vertex, so the curve avoids the toric fixed points.  These are all possible singularities of the compactified fiber.

To compute its arithmetic genus despite the singular toric ambient surface, let $D_P$ be the ample Cartier divisor determined by $P$ and take a toric resolution $\pi:Y\to X_P$.  The curve avoids the toric fixed points, which contain the singular locus, so its strict transform is isomorphic to $\overline C_\eta$ and has class $\pi^*D_P$.  Toric intersection formulas give $(\pi^*D_P)^2=2\operatorname{Area}(P)$ and $K_Y\cdot\pi^*D_P=-B(P)$, where $B(P)$ is the total boundary lattice length.  Adjunction and Pick's formula now give
\[
p_a(\overline C_\eta)=1+\frac{2\operatorname{Area}(P)-B(P)}2
=\#(P^\circ\cap N)=g_{\mathrm{cat}}(\mathcal X_P),
\]
where the last equality is Proposition~\ref{prop:toric-categorical-genus}.  Since the general fiber is irreducible by hypothesis, the normalization formula $p_a-g=\sum\delta$ proves the stated identity and identifies it with the drop in Corollary~\ref{cor:t-content-drop}.
\end{proof}

Under these hypotheses, the same calculation also recovers the residual-point genus formula:
\[
\begin{aligned}
g\bigl(\widetilde{\overline C_\eta}\bigr)
&=\#(P^\circ\cap N)-\sum_E m_E\binom{\ell_E}{2}\\
&=1+\#(R\cap P^\circ).
\end{aligned}
\]
Indeed, the crepant subdivision partitions $P$ into triangles over primitive $T$-cones and residual cones.  Each primitive $T$-triangle has area $\ell_E^2/2$, base lattice length $\ell_E$, and primitive radial sides, so Pick's formula gives $\binom{\ell_E}{2}$ interior lattice points.  The radial sides contain no other lattice points, leaving the origin and $R\cap P^\circ$ among the interior points.  Tveiten's Theorem~3.13 supplies the residual-point formula for general fibers of the maximally mutable family for any Fano polygon, without the edgewise simple-root and nonvanishing assumptions used here.

\section{Examples and explicit mirror checks}\label{sec:explicit-mirrors}

The general comparison is now established.  This section tests it on explicit curves and toric polygons, beginning with the motivating $X_{10}$ mirror.

Following \cite[Section 7]{OP}, we write $X_{k,d}$ for a surface with $k$ points of type $\frac13(1,1)$ and degree $d=K_X^2$.  Proposition~\ref{prop:third-rank} gives $g_{\mathrm{cat}}=k+1$, and the toric comparison predicts the same genus for the relevant maximally mutable mirrors.  Here we verify the genus directly for representative Laurent-polynomial curves constructed independently in the period calculations of Oneto--Petracci.  These computations provide evidence for the interpretation of categorical genus, rather than inputs to its formula.

We use representatives with $k=1,2,3,4,6$, the values treated in their explicit period computations.  Their omission of $k=5$ reflects the unavailable quantum-period calculation, not an absence of a toric degeneration or maximally mutable description.  Example~\ref{ex:five-point-toric} gives the comparison from the general formula but not a direct curve computation.  Section and polygon references follow the publisher version of Oneto--Petracci, where $X_{6,2}$ is entry~26 in Section~7.7.  Some cited entries specify only special subfamilies, which we distinguish from a general maximally mutable polynomial with the full polygon.  Genus below means the genus of the smooth projective model.  A nonconstant Laurent polynomial has finitely many critical values because it is constant on each irreducible component of its critical locus.  Thus its general fiber is smooth in the torus, and it remains to check the boundary of the compactification.

For $k=1$, the representative $X_{10}=X_{1,1/3}$ is entry~11 of \cite[Section 7.3, polygon 1]{OP}.  Their notation $a_{[p,q]}$ means the coefficient of the monomial $x^py^q$.  The displayed mirror specialization is
\[
f_{360}(x,y)=x^{-3}y^5(1+x+y^{-1})^{10}-2520,
\qquad a_{[4,3]}=360.
\]
The following proposition proves the genus directly, including a one-parameter Laurent-polynomial extension of this specialization.  We do not claim that the extension exhausts the normalized maximally mutable Laurent polynomial parameter space.

\begin{proposition}[The explicit $X_{10}$ genus-two model]\label{prop:x10-explicit-model}
Put $H=1+x+y^{-1}$ and define
\begin{equation}\label{eq:x10-factorized-family}
f_a=x^{-3}y^5H^{10}-2520+(a-360)(x^{-2}y^3H^6-60).
\end{equation}
Its Newton polygon is the triangle $P$ of Example~\ref{ex:x10-mirror-polygon}, its constant term is zero, and the coefficient of $x^4y^3$ in $f_a$ is $a$.  Set $\lambda=a-360$ and let $\mu=\eta+2520+60\lambda$ be the shifted fiber value.  For general $(\lambda,\mu)$, the fiber $f_{360+\lambda}=\eta$ is birational to
\begin{equation}\label{eq:x10-hyperelliptic}
v^2=-\mu(u+\lambda)\left(4u^3(u+\lambda)^2-\mu u^2(u+\lambda)+4\mu^2\right).
\end{equation}
Its smooth projective model has genus two.  The same conclusion holds for $a=360$ and general $\eta$.
\end{proposition}

\begin{proof}
The constant terms of $x^{-3}y^5H^{10}$ and $x^{-2}y^3H^6$ are, respectively,
\[
\frac{10!}{3!5!2!}=2520,\qquad \frac{6!}{2!3!1!}=60.
\]
Their coefficients of $x^4y^3$ are $10!/(7!2!1!)=360$ and $1$.  This proves the normalization.  The first summand has the three vertices of $P$ with coefficients one, and the support of the second lies strictly inside $P$, so the Newton polygon is unchanged.

On the fiber introduce
\[
u=x^{-1}y^2H^4,\qquad A=x^{-3}y^5H^{10},\qquad B=x^{-2}y^3H^6.
\]
Then $A=Bu$, so the fiber equation gives $A+\lambda B=\mu$ and $B(u+\lambda)=\mu$.  The monomial definitions also give
\[
x=\frac{u^3}{B^2},\qquad yH^2=\frac{u^2}{B}.
\]
Thus $B=\mu/(u+\lambda)$ and $x=u^3(u+\lambda)^2/\mu^2$.  Set $C=1+x$. The identity $H=C+y^{-1}$ gives
\[
C^2y^2+\left(2C-\frac{u^2}{B}\right)y+1=0.
\]
Its discriminant $\Delta=u^4/B^2-4Cu^2/B$ satisfies
\[
\frac{\mu^4}{u^2}\Delta
=-\mu(u+\lambda)\left(4u^3(u+\lambda)^2-\mu u^2(u+\lambda)+4\mu^2\right).
\]
Taking $v=(\mu^2/u)(2C^2y+2C-u^2/B)$ yields \eqref{eq:x10-hyperelliptic}.  Conversely,
\[
y=\frac{uv/\mu^2-2C+u^2/B}{2C^2}
\]
recovers $y$ rationally, and the preceding identities recover $x,u,A,B$ on dense open sets.  This proves birationality.  The minus sign is the one obtained from the discriminant. Replacing $v$ by $iv$ over $\C$ gives the opposite-sign quadratic twist.

For $\mu\ne0$, the right-hand side has degree six.  Put
\[
Q(u)=4u^3(u+\lambda)^2-\mu u^2(u+\lambda)+4\mu^2.
\]
For every fixed $\lambda$, the discriminant $\operatorname{disc}_u(Q)$, viewed as a polynomial in $\mu$, has leading term $-6912\mu^9$, so it is not identically zero.  Also $Q(-\lambda)=4\mu^2$.  Hence for every fixed $\lambda$ and general $\mu\ne0$, $Q$ is squarefree and has no root in common with $u+\lambda$.  The sextic on the right-hand side of \eqref{eq:x10-hyperelliptic} is then squarefree, so its double cover of $\Pj^1$ has six branch points. Riemann--Hurwitz gives genus two.  Since $\eta\mapsto\mu=\eta+2520+60\lambda$ is a translation for each fixed $\lambda$, the genus-two conclusion holds for every $a=360+\lambda$ and general $\eta$.
\end{proof}

\begin{example}[A quartic mirror for $X_{2,8/3}$]\label{ex:op-two-point}
For $k=2$, the one-parameter mirror of $X_{2,8/3}$ in \cite[Section 5.3 and Section 7.4, entry 15, polygon 13]{OP} gives a direct genus-three calculation:
\[
f_a(x,y)=\frac{(1+x+y)^3}{xy}+xy+a(x+y)-6.
\]
Writing $c=\eta+6$, clearing denominators gives $(1+x+y)^3+x^2y^2+a xy(x+y)-cxy=0$.  Its closure in $\Pj^2$ is the quartic
\begin{equation}\label{eq:op-quartic}
F_{a,c}=z(z+x+y)^3+x^2y^2+a xyz(x+y)-c xyz^2=0.
\end{equation}
To verify smoothness, take $a=0$, $c=1$.  On the line $z=0$, the curve meets only $[1:0:0]$ and $[0:1:0]$, and $F_z$ is nonzero at both.  In the affine chart $z=1$, put $G=(1+x+y)^3+x^2y^2-xy$.  At a singular point, the partial derivatives satisfy:
\[
G_x-G_y=(y-x)(2xy-1)=0.
\]
If the equation $2xy=1$ holds, the equations $G_x=G_y=0$ force $1+x+y=0$, but then $G$ equals $-1/4$.  Otherwise, set $x=y=t$.  The equations $G=G_x=0$ would give a common root of
\[
p(t)=t^4+8t^3+11t^2+6t+1,
\qquad q(t)=2t^3+12t^2+11t+3.
\]
Here $q$ is half the derivative of $p$, and the resultant is $\operatorname{Res}_t(p,q)=\operatorname{disc}(p)/16=-89\ne0$.  This checks all projective Jacobian equations.  Smoothness is open in the projective family \eqref{eq:op-quartic}, so its general member is a smooth, necessarily irreducible plane quartic.  The smooth plane quartic therefore has genus
\[
g=\frac{(4-1)(4-2)}2=3=g_{\mathrm{cat}}.
\]

\end{example}

\begin{example}[A nodal-cuspidal quintic for $X_{3,3}$]\label{ex:op-three-point}
For $k=3$, the specialization in \cite[Section 7.5, entry 20, polygon 14]{OP}, with $a_{[0,1]}=a_{[-1,0]}=0$ and $a_{[-1,1]}=1$, gives a direct genus-four calculation on a nodal-cuspidal quintic:
\[
f(x,y)=\frac{x}{y}\left(1+\frac1x+y\right)^2
+\frac yx\left(1+\frac1x+y\right)-2.
\]
Set $L=x+1+xy$ and $c=\eta+2$.  Clearing denominators gives $xL^2+y^2L-cx^2y=0$, whose quintic homogenization is
\[
H=x(xz+z^2+xy)^2+zy^2(xz+z^2+xy)-cx^2yz^2=0.
\]
The affine boundary points are $(0,0)$ and $(-1,0)$. Their respective nonzero derivatives are $H_x=1$ and $H_y=-c$ in the chart $z=1$.  At infinity the only points are $A=[0:1:0]$ and $B=[1:0:0]$.  In the chart $y=1$ at $A$, the quadratic part is $xz$, so $A$ is an ordinary node.  In the chart $x=1$ at $B$, the quadratic part is $(y+z)^2$, and restriction to its double tangent gives
\[
H(1,-z,z)=cz^3+z^4+z^5.
\]
For nonzero $c$, this is an $A_2$ cusp: after completing the square in the transverse coordinate $y+z$, the first nonzero term along the tangent has degree three.  Both singularities have $\delta=1$.  The generic-smoothness observation excludes further singularities for general $c$.  No coordinate line is a component, since
\[
H|_{z=0}=x^3y^2,\qquad H|_{x=0}=y^2z^3,\qquad
H|_{y=0}=xz^2(x+z)^2
\]
are all nonzero polynomials.  The quintic is also irreducible: a reduced reducible quintic without a boundary component would have, by B\'ezout, intersection multiplicity at least $1\cdot4=4$ between its components, exceeding the total $\delta$-invariant two.  The smooth projective model therefore has genus
\[
g=\frac{(5-1)(5-2)}2-1-1=4=g_{\mathrm{cat}}.
\]
This specialization retains the full polygon~14.  Thus, unlike the six-point family discussed below, it already realizes the genus predicted for the general maximally mutable Laurent polynomial.

\end{example}

\begin{example}[A genus-five $(4,4)$-curve for $X_{4,4/3}$]\label{ex:op-four-point}
For $k=4$, we use $X_{4,4/3}$, entry~24 of \cite[Section 7.6]{OP}, to compute the normalization genus of its mirror curve directly.  Its polygon~5 is the octagon with vertices $(\pm2,\pm1)$ and $(\pm1,\pm2)$, with all sign choices.  The four horizontal and vertical edges have binomial coefficients $1,2,1$.  Together with the interior coefficients specified in that entry, they give
\begin{equation}\label{eq:op-four-polynomial}
\begin{aligned}
f_a(x,y)={}&(x^2+x^{-2})(y+2+y^{-1})
+(y^2+y^{-2})(x+2+x^{-1})\\
&+(a+8)(x+x^{-1})(y+y^{-1})
+(2a+14)(x+x^{-1}+y+y^{-1}).
\end{aligned}
\end{equation}
Multiplying $f_a-\eta$ by $x^2y^2$ and bihomogenizing gives a curve $C$ of bidegree $(4,4)$ in $\Pj^1\times\Pj^1$.  Since the canonical bundle of $\Pj^1\times\Pj^1$ is $\mathcal O(-2,-2)$ and the intersection pairing satisfies $(a,b)\cdot(c,d)=ad+bc$, adjunction gives
\[
2p_a(C)-2=(4,4)\cdot(2,2)=16,
\qquad p_a(C)=9.
\]
On the line $x=0$, its equation restricts to $y(y+1)^2$. Inversion of either coordinate and interchange of $x,y$ give the other boundary charts.  The four corners are smooth, with local linear part $x+y$ at $(0,0)$.  The only remaining boundary points are
\[
(0,-1),\quad(\infty,-1),\quad(-1,0),\quad(-1,\infty).
\]
At the first, in coordinates $u=x$, $v=y+1$, the quadratic part is
\[
-v^2-(\eta+4a+24)u^2.
\]
It has distinct tangent lines when the coefficient $\eta+4a+24$ is nonzero. The symmetries give the same conclusion at all four points.  They are therefore ordinary nodes, and a general fiber has no additional singularities in the torus.

To prove irreducibility, we use the following change of variables.  Put $U=x+x^{-1}+2$, $V=y+y^{-1}+2$ and $b=\eta+4a+24$.  The fiber equation becomes
\[
UV(U+V+a)=b.
\]
Its plane cubic closure is smooth for general $(a,b)$.  The two quadratic equations recovering $x$ and $y$ have discriminants $U(U-4)$ and $V(V-4)$.  On the cubic, the first has simple zeros at the two affine points where $U$ equals $4$. The third intersection of the projective line $U=4$ with the cubic is the point at infinity $[0:1:0]$.  At each of these two affine points the second discriminant is a unit for general parameters, and the analogous statement holds with $U,V$ interchanged.  A square in a function field has even valuation at every point.  These simple zeros show that neither discriminant, nor their product, is a square, so their square classes are independent.  Adjoining both square roots gives a degree-four field extension, proving irreducibility of the $(4,4)$ curve.  Consequently its normalization has
\[
g=9-4=5=g_{\mathrm{cat}}.
\]

\end{example}

\begin{example}[The six-point mirror family for $X_{6,2}$]\label{ex:op-six-point}
We use the family with the specified Newton polygon to compute the genus of a general fiber and distinguish it from the expression printed in \cite[Section 7.7, entry 26]{OP}.  For $k=6$, the family associated with polygon~9 in that entry has Newton polygon
\begin{equation}\label{eq:op-six-polygon}
P_6=\operatorname{conv}\{(1,-2),(2,-1),(1,1),(-1,2),(-2,1),(-1,-1)\}.
\end{equation}

The expression printed in this entry omits $xy$, although $(1,1)$ is a vertex of the specified Newton polygon.  The opposite vertex monomial $x^{-1}y^{-1}$ is present with coefficient one.  Restoring $xy$ with coefficient one therefore adds $2$ to the contribution of the antipodal vertex pairs to the constant term of $f^2$, changing that vertex contribution from $4$ to $6$.  This is consistent with the printed coefficient of $t^2$ in the period and supports a typographical omission, but does not by itself verify the full period.  The computations below use the family with the specified Newton polygon.

For a nondegenerate curve, the adjunction formula $p_a(C)=1+(C^2+C\cdot K)/2$ on a smooth toric compactification gives the number of interior Newton lattice points \cite[Section 1, Theorem 1]{KhovanskiiGenus}.

With vertex coefficients normalized to one, write the family as
\[
f_{\boldsymbol a}=\frac{x}{y^2}+\frac{x^2}{y}+xy+\frac{y^2}{x}
+\frac y{x^2}+\frac1{xy}
+a_1x+a_2y+a_3\frac yx+a_4\frac1x+a_5\frac1y+a_6\frac xy.
\]
Every edge of $P_6$ has width one and height three. Hence all six cones are residual $\frac13(1,1)$ cones, and there are no $T$-mutability constraints.  The six edge determinants are all three, giving area nine, while all six edges are primitive, so the total lattice boundary length, equivalently the boundary lattice-point count, is six.  Pick's theorem yields
\[
\#(P_6^\circ\cap\Z^2)=9-\frac62+1=7.
\]
Each edge polynomial is a primitive binomial with a simple root.  The closure on the toric surface of the normal fan therefore meets the boundary transversely and avoids its fixed points. For general parameters and fiber value it is also smooth in the torus.  Toric adjunction gives arithmetic genus $1+(18-6)/2=7$.  This smooth ample curve is connected, hence irreducible, so its geometric genus is seven.

\end{example}

\begin{proposition}[Explicit mirror check]\label{prop:explicit-mirror-check}
For the $X_{1,1/3}$ representative of Proposition~\ref{prop:x10-explicit-model} and the representatives of $X_{2,8/3}$, $X_{3,3}$, $X_{4,4/3}$, and $X_{6,2}$ in Examples~\ref{ex:op-two-point}--\ref{ex:op-six-point}, the genera of general fibers agree with categorical genus as follows:
\[
\begin{array}{c|c|c}
k&g_{\mathrm{cat}}&g(\text{general fiber in the subfamily})\\\hline
1&2&2\\2&3&3\\3&4&4\\4&5&5\\6&7&7
\end{array}
\]
\end{proposition}

\begin{proof}
  The preceding compactifications give the displayed genera directly on the indicated representatives.  Theorem~\ref{cor:genus-match} gives the same genera for general maximally mutable Laurent polynomials in their mutation classes, while Proposition~\ref{prop:third-rank} gives the categorical values.
\end{proof}

These checks give direct geometric realizations of the predicted genera on mirror models constructed independently of the Euler pairing.

\begin{example}[Cavey--Prince polygons]\label{ex:cavey-prince}
To test other residual singularities, consider the polygons
\[
P=\operatorname{conv}\{(0,1),(-k_1,-1),(k_2,-1)\}
\]
of Cavey--Prince \cite[Proposition 7.1 and Remark 7.2]{CP}, with $(k_1,k_2)=(3,6)$ or $(5,5)$.  Their toric surfaces admit $\Q$-Gorenstein deformations to $X_{k_1+k_2}\subset\Pj(1,1,k_1,k_2)$, with residual basket $\frac1{k_1}(1,1),\frac1{k_2}(1,1)$.  For direct genus calculations we use the binomial specialization
\[
f_{k_1,k_2}(x,y)=y+\frac{(1+x)^{k_1+k_2}}{x^{k_1}y}.
\]
This polynomial has Newton polygon $P$. We obtain it by choosing the binomial coefficients along the horizontal edge in the polygon construction of \cite{CP}.

In primitive coordinates on its three edges, these polynomials have edge polynomials $(1+T)^{k_1+k_2}$ on the horizontal edge and $1+T$ on the $k_1$-edge.  On the $k_2$-edge they have $1+T^2$ when $k_2=6$ and $1+T$ when $k_2=5$.  Thus both are normalized maximally mutable members with $T$-binomial edge data \cite[Propositions 3.7 and 3.9]{CKPT}.  All interior residual coefficients are zero, so they are special points of the parameter spaces, not general members.

For $X_9\subset\Pj(1,1,3,6)$, the $\frac13(1,1)$ point contributes $1$, while the $\frac16(1,1)$ point has $w=2$, $\ell=3$, and contributes $2$.  Hence $g_{\mathrm{cat}}=1+1+2=4$.  On $f_{3,6}=\eta$, set $v=x^2(2y-\eta)$.  Clearing denominators gives
\[
v^2=P_{10}(x):=\eta^2x^4-4x(1+x)^9.
\]
For Cavey--Prince's $X_{10}^{\mathrm{CP}}\subset\Pj(1,1,5,5)$, each of the two $\frac15(1,1)$ points contributes $2$, so $g_{\mathrm{cat}}=1+2+2=5$.  On $f_{5,5}=\eta$, the substitution $v=x^3(2y-\eta)$ gives
\[
v^2=P_{11}(x):=\eta^2x^6-4x(1+x)^{10}.
\]
At $\eta=1$, direct polynomial division gives $\gcd(P_{10},P_{10}')=\gcd(P_{11},P_{11}')=1$.  Squarefreeness is open in $\eta$, and the degrees are $10$ and $11$, giving genera $4$ and $5$.  These curve calculations check the genus on the stated special members.  Theorem~\ref{cor:genus-match} gives genera $4$ and $5$ for general fibers of the corresponding maximally mutable families.  The genus-$4$ curve tests the residual width $w_0=2$ case of Lemma~\ref{lem:age-count}.  Cavey--Prince's cascade polygons and Laurent polynomials provide analogous examples with a single singular point for arbitrary $k$ \cite[Table 8.1]{CP}.
\end{example}

\begin{example}[Separate residual and smoothable cones]\label{ex:petracci-drop}
This example exhibits residual and smoothable contributions simultaneously, so the categorical genus drops under a general $\Q$-Gorenstein deformation.

Let $X_P$ be the toric del Pezzo surface with polygon
\[
P=\operatorname{conv}\{(2,1),(1,2),(-1,2),(-2,-1),(-1,-2),(1,-2)\}.
\]
Petracci \cite[Proposition 2.1]{PetracciKmoduli} proves that it is K-polystable and has basket
\[
2\times\tfrac13(1,1),\qquad
2\times\tfrac14(1,1),\qquad
2\times\tfrac15(1,2).
\]
The first and third types are $\Q$-Gorenstein rigid.  The two $\frac14(1,1)$ points admit simultaneous global $\Q$-Gorenstein smoothings: the global $\Q$-Gorenstein deformation functor has hull $\C[\![t_1,t_2]\!]$, with independent smoothing parameters, and Petracci realizes it by an explicit family \cite[Section 2.2, equation (5), Proposition 2.3(B), and Section 2.5]{PetracciKmoduli}.  Write $X_t$ for a sufficiently general fiber of this family and $\mathcal X_t$ for its canonical stack.  Its basket is
\[
2\times\tfrac13(1,1)+2\times\tfrac15(1,2).
\]
Theorem~\ref{thm:cyclic-rank} assigns local genus corrections $1$, $1$, and $2$ to $\frac13(1,1)$, $\frac14(1,1)$, and $\frac15(1,2)$, respectively: the relevant gcds are $1$, $2$, and $1$.  Consequently the categorical genera are
\[
g_{\mathrm{cat}}(\mathcal X_P)=1+2\cdot1+2\cdot1+2\cdot2=9,
\qquad
g_{\mathrm{cat}}(\mathcal X_t)=1+2\cdot1+2\cdot2=7.
\]
Figure~\ref{fig:petracci-hexagon} displays the two smoothable contributions and the six residual points.
\begin{figure}[htbp]
\centering
\begin{tikzpicture}[x=1.08cm,y=1.08cm,font=\small]
  \fill[blue!9] (2,1)--(1,2)--(-1,2)--(-2,-1)--(-1,-2)--(1,-2)--cycle;
  \fill[orange!22] (0,0)--(1,2)--(-1,2)--cycle;
  \fill[orange!22] (0,0)--(-1,-2)--(1,-2)--cycle;
  \foreach \horizontal/\vertical in {2/1,1/2,-1/2,-2/-1,-1/-2,1/-2}
    \draw[gray!70,dashed] (0,0)--(\horizontal,\vertical);
  \draw[thick] (2,1)--(1,2)--(-1,2)--(-2,-1)--(-1,-2)--(1,-2)--cycle;
  \foreach \horizontal in {-1,1}
    \foreach \vertical in {-1,0,1}
      \fill[blue!70!black] (\horizontal,\vertical) circle (2pt);
  \foreach \vertical in {-1,1}
    \fill[orange!65!black] (0,\vertical) circle (2.6pt);
  \fill (0,0) circle (2.2pt) node[below right] {$0$};
  \node[above] at (0,2.6) {$T$: $(w,\ell)=(2,2)$};
  \node[below] at (0,-2.6) {$T$: $(w,\ell)=(2,2)$};
  \node[above right] at (2,1) {$(2,1)$};
  \node[above right] at (1,2) {$(1,2)$};
  \node[above left] at (-1,2) {$(-1,2)$};
  \node[below left] at (-2,-1) {$(-2,-1)$};
  \node[below left] at (-1,-2) {$(-1,-2)$};
  \node[below right] at (1,-2) {$(1,-2)$};
  \node[anchor=west,align=left] at (3.1,0)
    {$2$ orange points: $T$-content\\[3pt]
     $6$ blue points $+$ origin\\[3pt]
     $g_{\mathrm{cat}}=9\longrightarrow g_{\mathrm{cat}}^{\mathrm{qG}}=7$};
\end{tikzpicture}
\caption{Petracci's hexagon, with its two $T$-cones shaded orange.  Each contains one interior lattice point. The four residual cones contain six points in total.  Keeping these six points and the origin gives residual categorical genus $7$, compared with categorical genus $9$ for the original toric stack.}
\label{fig:petracci-hexagon}
\end{figure}

Petracci \cite[Section 3.2]{PetracciKmoduli} prints a six-parameter maximally mutable Laurent polynomial family, but its display omits two determined coefficients.  To verify them, write the top height-one slice as $f_1(x)=x^2+a_1x+d_++c_2x^{-1}$, where $d_+$ is the coefficient of $y$.  Mutability with factor $1+x$, as in \cite[Proposition 3.7]{CKPT}, requires $xf_1(x)$ to be divisible by $1+x$, and evaluation at $x=-1$ gives $d_+=a_1+c_2-1$.  The bottom slice $f_{-1}(x)=x^{-2}+a_2x^{-1}+d_-+c_1x$ similarly requires $xf_{-1}(x)$ to be divisible by $1+x$, which gives $d_-=a_2+c_1-1$.  Setting the omitted coefficients to zero would impose $a_1+c_2=1$ and $a_2+c_1=1$, rather than leave the six parameters independent.  The normalized family with six independent parameters is therefore
\[
\begin{aligned}
\widetilde f={}&x^2y+x^{-2}y^{-1}+(x+2+x^{-1})(y^2+y^{-2})\\
&+a_1xy+a_2x^{-1}y^{-1}+b_1x+b_2x^{-1}\\
&+c_1xy^{-1}+c_2x^{-1}y\\
&+(a_1+c_2-1)y+(a_2+c_1-1)y^{-1}.
\end{aligned}
\]
Petracci's specialization $a_1=a_2=1$ and $b_1=b_2=c_1=c_2=0$ remains valid, since both added coefficients vanish.

Let $C_\eta$ be the closure of $\widetilde f=\eta$ in the toric surface defined by the normal fan of $P$.  Near the top boundary point, use local coordinates $u=x+1$ and $z=y^{-1}$.  After clearing the pole along the boundary, a local equation is
\[
F_+(u,z)=u^2+(2-a_1+c_2)uz+(b_1+b_2+\eta)z^2+O((u,z)^3).
\]
Near the bottom boundary point, use $u=x+1$ and $z=y$.  The corresponding local equation is
\[
F_-(u,z)=u^2+(a_2-c_1-2)uz+(b_1+b_2+\eta)z^2+O((u,z)^3).
\]
The discriminants of their tangent cones are
\[
\Delta_+=(2-a_1+c_2)^2-4(b_1+b_2+\eta),\qquad
\Delta_-=(a_2-c_1-2)^2-4(b_1+b_2+\eta).
\]
For general parameters and $\eta$, both discriminants are nonzero, so each tangent cone consists of two distinct lines.  The two boundary singularities are ordinary nodes, each with $\delta=1$.

There are no other singularities for a general member.  The top and bottom edge polynomials are $x^{-1}(1+x)^2$, with residual factor $1$.  The other four edges of $P$ are primitive and have binomial face polynomials with simple roots.  The nonzero vertex coefficients keep the curve away from the toric fixed points.  In the torus, the parameter directions $1$, $x$, and $y(x+1)$ span the first jets wherever $x\ne-1$, so a general member is smooth there.  On $x=-1$, one has $\widetilde f=-(b_1+b_2)$, and a general value $\eta$ avoids this line.  The moving linear system is birational on $x\ne-1$, since the ratios of these sections recover $x$ and $y$.  Its base locus is finite, so Bertini's irreducibility theorem gives an irreducible general member.

Pick's theorem gives $\operatorname{Area}(P)=12$ and eight boundary lattice points, hence nine interior lattice points.  Toric adjunction gives $p_a(C_\eta)=9$.  The two nodes give $\sum_p\delta_p=2$, exactly the case $m=1$, $\ell=2$ of Proposition~\ref{prop:boundary-delta}.  The normalization formula yields
\[
g(\widetilde C_\eta)=p_a(C_\eta)-\sum_p\delta_p=9-2=7.
\]

The two width-two, height-two cones are pure $T$-cones.  The remaining four residual cones contribute $1,1,2,2$ interior residual points.  Thus the maximally mutable Laurent polynomial genus formula \cite[Theorem 3.13]{CKPT} gives
\[
\boxed{g_{\mathrm{cat}}(\mathcal X_P)=9,\qquad
g_{\mathrm{cat}}(\mathcal X_t)=7,\qquad
g_{\mathrm{mut}}(P)=7.}
\]
The drop $9-7=2$ is exactly the contribution of the two smoothed $\frac14(1,1)$ points.  The resulting basket is the residual basket of a sufficiently general $\Q$-Gorenstein deformation, as described in Lemma~\ref{lem:generic-residual-basket}.
\end{example}

\begin{example}[A mixed cone with residual width $2$]\label{ex:mixed-cone-delta}
Consider the polygon
\[
\boxed{P=\operatorname{conv}\{(-1,3),(4,3),(0,-1)\}.}
\]
Figure~\ref{fig:mixed-cone-example} shows the polygon, its two residual subdivisions, and the mutation described below.

\begin{figure}[!htbp]
\centering
\begin{minipage}[t]{0.49\textwidth}
\centering
\begin{tikzpicture}[x=0.82cm,y=0.82cm,font=\small]
  \useasboundingbox (-2.2,-2.1) rectangle (5.8,5.7);
  \node at (2.4,5.35) {(a) Original polygon $P$};
  \draw[step=1,gray!20,very thin] (-1.3,-1.3) grid (4.3,3.3);
  \draw[gray!70,dashed] (0,0)--(-1,3) (0,0)--(4,3);
  \draw[thick] (-1,3)--(4,3)--(0,-1)--cycle;
  \draw[line width=1.1pt] (-1,3)--(4,3);
  \foreach \horizontal/\vertical in {0/0,0/1,0/2,1/1,1/2,2/2}
    \fill (\horizontal,\vertical) circle (2pt);
  \foreach \horizontal in {-1,1,2,4}
    \fill (\horizontal,3) circle (2pt);
  \node[above left] at (-1,3) {$(-1,3)$};
  \node[above left] at (1,3) {$(1,3)$};
  \node[above right] at (2,3) {$(2,3)$};
  \node[right] at (4,3) {$(4,3)$};
  \node[below] at (0,-1) {$(0,-1)$};
  \node[below left] at (0,0) {$0$};
  \node at (2,4.35) {$\frac1{15}(1,4),\quad(w,\ell)=(5,3)$};
  \node[align=left,anchor=west] at (3.4,0.35) {$A_3$\\$(w,\ell)=(4,1)$};
  \node[rotate=76] at (-1.05,0.85) {smooth};
  \node at (2.4,-1.8) {$5=1\cdot3+2$; six interior points};
\end{tikzpicture}
\end{minipage}\hfill
\begin{minipage}[t]{0.49\textwidth}
\centering
\begin{tikzpicture}[x=0.82cm,y=0.82cm,font=\small]
  \useasboundingbox (-2.2,-2.1) rectangle (5.8,5.7);
  \node at (2.4,5.35) {(b) First crepant subdivision};
  \fill[orange!22] (0,0)--(-1,3)--(2,3)--cycle;
  \fill[blue!9] (0,0)--(2,3)--(4,3)--cycle;
  \draw[step=1,gray!20,very thin] (-1.3,-1.3) grid (4.3,3.3);
  \draw[thick] (-1,3)--(4,3)--(0,-1)--cycle;
  \draw[gray!70,dashed] (-1,3)--(0,0)--(4,3);
  \draw[thick] (0,0)--(2,3);
  \foreach \horizontal/\vertical in {0/1,0/2,1/2}
    \fill[orange!65!black] (\horizontal,\vertical) circle (2pt);
  \foreach \horizontal/\vertical in {1/1,2/2}
    \fill[blue!70!black] (\horizontal,\vertical) circle (2.5pt);
  \fill (0,0) circle (2.2pt) node[below left] {$0$};
  \fill (2,3) circle (2pt) node[above] {$(2,3)$};
  \node[text=orange!65!black,align=center] at (0.1,4.15) {$T$-part\\width $3$};
  \node[text=blue!70!black,align=center] at (3.8,4.15) {residue\\width $2$};
  \node[text=blue!70!black,anchor=west,align=left] at (4.45,1.6) {$(1,1)$\\[3pt]$(2,2)$};
  \node at (2.4,-1.8) {$\#(R_1\cap P^\circ)=2$};
\end{tikzpicture}
\end{minipage}

\medskip
\begin{minipage}[t]{0.49\textwidth}
\centering
\begin{tikzpicture}[x=0.82cm,y=0.82cm,font=\small]
  \useasboundingbox (-2.2,-2.1) rectangle (5.8,5.7);
  \node at (2.4,5.35) {(c) Second crepant subdivision};
  \fill[blue!9] (0,0)--(-1,3)--(1,3)--cycle;
  \fill[orange!22] (0,0)--(1,3)--(4,3)--cycle;
  \draw[step=1,gray!20,very thin] (-1.3,-1.3) grid (4.3,3.3);
  \draw[thick] (-1,3)--(4,3)--(0,-1)--cycle;
  \draw[gray!70,dashed] (-1,3)--(0,0)--(4,3);
  \draw[thick] (0,0)--(1,3);
  \foreach \horizontal/\vertical in {1/1,1/2,2/2}
    \fill[orange!65!black] (\horizontal,\vertical) circle (2pt);
  \foreach \vertical in {1,2}
    \fill[blue!70!black] (0,\vertical) circle (2.5pt);
  \fill (0,0) circle (2.2pt) node[below left] {$0$};
  \fill (1,3) circle (2pt) node[above] {$(1,3)$};
  \node[text=blue!70!black,align=center] at (-0.35,4.15) {residue\\width $2$};
  \node[text=orange!65!black,align=center] at (3.2,4.15) {$T$-part\\width $3$};
  \node[text=blue!70!black,anchor=west,align=left] at (4.45,1.6) {$(0,1)$\\[3pt]$(0,2)$};
  \node at (2.4,-1.8) {$\#(R_2\cap P^\circ)=2,\quad R_1\ne R_2$};
\end{tikzpicture}
\end{minipage}\hfill
\begin{minipage}[t]{0.49\textwidth}
\centering
\begin{tikzpicture}[x=1cm,y=1cm,font=\small]
  \useasboundingbox (-1.3,-1.1) rectangle (5.68,4.3);
  \node at (2.15,4.0) {(d) Mutation and lattice equivalence};
  \node at (0,3.1) {$P\rightsquigarrow P'$};
  \begin{scope}[x=0.58cm,y=0.58cm]
    \draw[step=1,gray!20,very thin] (-1.2,-1.2) grid (1.2,3.2);
    \draw[thick] (-1,3)--(1,3)--(1,-1)--(0,-1)--cycle;
    \foreach \horizontal/\vertical in {-1/3,1/3,1/-1,0/-1}
      \fill (\horizontal,\vertical) circle (1.8pt);
    \node[above left] at (-1,3) {$(-1,3)$};
    \node[above right] at (1,3) {$(1,3)$};
    \node[below right] at (1,-1) {$(1,-1)$};
    \node[below left] at (0,-1) {$(0,-1)$};
  \end{scope}
  \draw[->,thick] (1.2,1.65)--(2.25,1.65) node[midway,above] {$M$};
  \begin{scope}[x=0.58cm,y=0.58cm,xshift=3.45cm]
    \draw[step=1,gray!20,very thin] (-1.2,-1.2) grid (1.2,6.2);
    \draw[thick] (1,0)--(0,-1)--(-1,2)--(-1,6)--cycle;
    \foreach \horizontal/\vertical in {1/0,0/-1,-1/2,-1/6}
      \fill (\horizontal,\vertical) circle (1.8pt);
    \node[right] at (1,0) {$(1,0)$};
    \node[below] at (0,-1) {$(0,-1)$};
    \node[left] at (-1,2) {$(-1,2)$};
    \node[left] at (-1,6) {$(-1,6)$};
    \node[right] at (0.25,4.8) {$P_6^{(4)}$};
  \end{scope}
\end{tikzpicture}
\end{minipage}
\caption{The mixed-cone example of Example~\ref{ex:mixed-cone-delta}.  Panel (a) shows the polygon with a width-$3$ $T$-part and width-$2$ residual part on its top edge.  Panels (b) and (c) show the two distinct residual subdivisions, each with two interior points.  Panel (d) displays the mutation to a polygon $\mathrm{GL}(2,\Z)$-equivalent to $P_6^{(4)}$.}
\label{fig:mixed-cone-example}
\end{figure}

In this order its edges have absolute determinants $15,4,1$, widths $5,4,1$, and heights $3,1,1$.  The corresponding cones have types $\frac1{15}(1,4)$, $\frac14(1,3)=A_3$, and smooth, respectively.  The top edge has $w=\gcd(15,5)=5$ and its width decomposes as $5=1\cdot3+2$. Since the ratio $(4+1)/5$ equals one, its residue is $\frac16(1,1)$.  The $A_3$ cone has $m=4$, $w_0=0$, $\ell=1$, so contributes neither genus nor genus drop.

The interior lattice points are
\[
P^\circ\cap N=\{(0,0),(0,1),(0,2),(1,1),(1,2),(2,2)\}.
\]
The categorical and residual formulas give
\[
g_{\mathrm{cat}}(\mathcal X_P)=1+\frac{15-5}{2}=6,
\qquad
g_{\mathrm{cat}}^{\mathrm{qG}}(X_P)=1+\frac{2(3-1)}2=3,
\]
with drop $\binom32=3$.  The primitive generators on the top edge are $(-1,3),(1,3),(2,3),(4,3)$.  Cutting off its width-$3$ $T$-part at $(2,3)$ or at $(1,3)$ gives, respectively, residual triangles
\[
\begin{aligned}
\operatorname{conv}\{0,(2,3),(4,3)\},&\quad R\cap P^\circ=\{(1,1),(2,2)\},\\
\operatorname{conv}\{0,(-1,3),(1,3)\},&\quad R\cap P^\circ=\{(0,1),(0,2)\}.
\end{aligned}
\]
Thus Definition~\ref{def:residual-lattice-points} gives distinct residual sets, each with two interior lattice points.

The divisibilities of the $y^3$, $y^2$, and $y$ slices verify the mutation with weight $(0,-1)$ and factor $[0,(1,0)]$.  Its defining height-slice construction gives
\[
P'=\operatorname{conv}\{(-1,3),(1,3),(1,-1),(0,-1)\}.
\]
The unimodular matrix $\left(\begin{smallmatrix}-1&0\\3&1\end{smallmatrix}\right)$ maps its vertex set to $\{(1,0),(0,-1),(-1,2),(-1,6)\}$.  This is precisely Cavey--Prince's $P_6^{(4)}$ \cite[Section 4.1]{CP}.  Their Proposition~4.2 constructs a general deformation in the family $X_6^{(4)}$, the blow-up of $\Pj(1,1,6)$ at four general smooth points \cite[Definition 2.7]{CP}. Mutation preserves the generic $\Q$-Gorenstein deformation class \cite[Theorem 3]{ACC+}.  The dual polygon has vertices $(0,-1/3),(-1,1),(4,1)$ and area $10/3$.  Hence the toric anticanonical degree is
\[
K_{X_P}^2=2\operatorname{Area}(P^\vee)=\frac{20}{3}.
\]
This agrees with the cascade model, since $X_6^{(4)}$ is the blow-up of $\Pj(1,1,6)$ at four points and $K_{\Pj(1,1,6)}^2-4=64/6-4=20/3$ \cite[Remark 2.8]{CP}.  Its sole remaining singularity is $\frac16(1,1)$.  Theorem~\ref{thm:cyclic-rank} gives categorical genus $1+2=3$ for its canonical stack.

An explicit normalized maximally mutable Laurent polynomial for $P$ is
\[
\boxed{\begin{aligned}
f={}&x^{-1}y^3(1+x)^3(1+cx+x^2)\\
&+(1+x)^2(a+4x)y^2+(1+x)(b+6x)y+4x+y^{-1}.
\end{aligned}}
\]
The $y^3$, $y^2$, and $y$ slices are, respectively,
\[
x^{-1}(1+x)^3(1+cx+x^2),\qquad
(1+x)^2(a+4x),\qquad
(1+x)(b+6x).
\]
Thus the slices of heights $3$, $2$, and $1$ relative to the top edge have the required divisors $(1+x)^3$, $(1+x)^2$, and $1+x$.  The top factor $(1+x)^3$ is the $T$-part, and $1+cx+x^2$ is the residual quadratic.  Along the right edge, the coefficients $1,4,6,4,1$ give $y^{-1}(1+xy)^4$.  The remaining edge is primitive.  These are the edge-slice conditions for the mutation and maximal mutability \cite[Propositions 3.7 and 3.9]{CKPT}.  All exponents lie in $P$, the three vertex coefficients are $1$, and the constant coefficient is $0$.  The independent coefficients at the residual points $(0,1),(0,2),(0,3)$ of the second subdivision are $b,a,c+3$, respectively.

Let $\overline C_\eta$ be the toric closure of $f=\eta$ in the normal toric surface $Y_P$ defined by the normal fan of $P$.  All three vertex coefficients are nonzero, so the curve avoids the toric fixed points and lies in the smooth locus of $Y_P$.  It is a Cartier divisor of class $D_P$.  The toric intersection formulas give $D_P^2=2\operatorname{Area}(P)=20$ and $K_{Y_P}\mathbin{\cdot}D_P=-\#(\partial P\cap N)=-10$.  Adjunction applies because the curve lies in the smooth locus, and Pick's formula then gives
\[
p_a(\overline C_\eta)=1+\frac12D_P\mathbin{\cdot}(D_P+K_{Y_P})
=\#(P^\circ\cap N)=6.
\]
The Newton triangle is Minkowski indecomposable: any nontrivial summands of a triangle are homothetic triangles, while the lattice edge lengths $5,4,1$ have greatest common divisor $1$.  A nontrivial factorization would induce such a decomposition, so the curve is irreducible.  Its torus part is smooth for general $\eta$.

At the right boundary use $s=xy$ and $z=y$, with boundary $z=0$.  The local equation $z(f-\eta)=0$ restricts to $(1+s)^4$, and its transverse derivative at $(s,z)=(-1,0)$ is $a-b-c-\eta-1$.  For general parameters the point is smooth, with order-$4$ tangency and $\delta=0$.  At the top boundary put $t=x+1$, $u=y^{-1}$.  The equation $xu^3(f-\eta)=0$ has no terms of total degree below $3$, and its cubic part is
\[
\boxed{(2-c)t^3+(4-a)t^2u+(6-b)tu^2+(4+\eta)u^3.}
\]
This binary cubic has distinct roots for general parameters: at $a=b=c=\eta=0$, its discriminant is $-448$.  The point is therefore an ordinary triple point with $\delta=3$.  If $c^2\ne4$, the residual quadratic has two simple roots different from $-1$, giving transverse intersections with the top boundary.  The primitive third edge has a binomial edge polynomial with a simple root, so its intersection is transverse.  Since the vertex coefficients are nonzero, the curve avoids the toric fixed points.  These checks account for all boundary points, and the argument above for general $\eta$ excludes singularities in the torus.  Consequently the normalization has genus
\[
g\bigl(\widetilde{\overline C_\eta}\bigr)=6-3=3,
\qquad
\boxed{(g_{\mathrm{cat}},g_{\mathrm{cat}}^{\mathrm{qG}},g_{\mathrm{mut}}(P))=(6,3,3).}
\]
The ordinary triple point realizes $\delta=1\binom32=3$ in Proposition~\ref{prop:boundary-delta}, exactly the contribution of the width-$3$, height-$3$ $T$-part of this mixed cone.
\end{example}

\begin{example}[The five-point surface $X_{5,5/3}$]\label{ex:five-point-toric}
Corti--Heuberger give a toric $\Q$-Gorenstein degeneration of $X_{5,5/3}$ with polygon $P_7$ \cite[Section 7 and Table 4, polygon 7]{CH}:
\[
P_7=\operatorname{conv}\{(2,1),(1,2),(-1,2),(-2,1),(-2,-1),(-1,-2),(1,-1)\}.
\]
The polygon has area $23/2$ and boundary lattice length $9$, so Pick's formula gives $\#(P_7^\circ\cap N)=8$.  The edge cones consist of five $\frac13(1,1)$ points with $(w,\ell)=(1,3)$ and two $\frac14(1,1)$ points with $(w,\ell)=(2,2)$.  Thus the categorical genus of $\mathcal X_{P_7}$ is
\[
g_{\mathrm{cat}}(\mathcal X_{P_7})=1+5\cdot1+2\cdot1=8.
\]
The two $T$-singularities each contribute a drop of $\binom22=1$ under smoothing, while the five residual points persist.  Corollaries~\ref{cor:t-content-drop} and~\ref{cor:qg-generic-mirror-genus} give
\[
g_{\mathrm{cat}}^{\mathrm{qG}}(X_{P_7})=8-2=6,
\qquad \boxed{g_{\mathrm{mut}}(P_7)=6=5+1.}
\]
This gives the general-formula genus comparison for $k=5$, not a direct curve computation or the missing Oneto--Petracci quantum-period calculation.
\end{example}

\begin{example}[The $X_{10}$ mirror polygon: a drop of $34$]\label{ex:x10-toric-drop}
Return to the triangle $P$ in Example~\ref{ex:x10-mirror-polygon}, and let $X_P$ be the toric surface defined by the fan over its edges, with canonical stack $\mathcal X_P$.  Its three edge cones have width--height pairs
\[
(w,\ell)=(10,3),\quad(10,5),\quad(10,2).
\]
Theorem~\ref{thm:cyclic-rank} and Proposition~\ref{prop:toric-categorical-genus} give
\[
g_{\mathrm{cat}}(\mathcal X_P)=1+10+20+5=36=\#(P^\circ\cap\Z^2).
\]
The width decompositions are $10=3\cdot3+1$, $10=2\cdot5$, and $10=5\cdot2$, so the only residue has width one and height three, of type $\frac13(1,1)$.  Consequently Proposition~\ref{prop:qg-residual-genus} and Corollary~\ref{cor:t-content-drop} give
\[
g_{\mathrm{cat}}^{\mathrm{qG}}(X_P)=2,\qquad
36-2=\frac12\bigl(3\cdot3\cdot2+2\cdot5\cdot4+5\cdot2\cdot1\bigr)=34.
\]
Figure~\ref{fig:x10-subdivision} displays these $34$ discarded points.  The genus $36$ belongs to the canonical stack of the toric degeneration, not to $\cX_{10}$, the Oneto--Petracci hypersurface in $\Pj(1,2,3,5)$, whose categorical genus is already two.
\end{example}

Together, the examples distinguish the categorical genus of the original toric stack, the residual categorical genus after general $\Q$-Gorenstein deformation, and the genus of the maximally mutable mirror fiber.  Direct curve calculations provide independent checks where available, while the polygon examples show how smoothable and residual contributions combine.

\begin{remark}[The scope of the mirror comparison]\label{rem:johnson-kollar-mirror}
Theorem~\ref{thm:intro-mirror} and Corollary~\ref{cor:qg-generic-mirror-genus} concern maximally mutable Laurent polynomial mirrors associated with toric $\Q$-Gorenstein degenerations.  The Johnson--Koll\'ar surfaces of Example~\ref{ex:johnson-kollar-categorical} have empty anticanonical system and therefore admit no such degeneration \cite[Section 1.1]{GR}.  Corti--Gugiatti construct a hyperelliptic pencil whose general fiber has genus $3k+1$ and prove that its period is the regularized $I$-function of the family \cite[Remarks 1.4 and 1.7 and Theorem 1.5]{CG}.  The categorical formula gives $g_{\mathrm{cat}}=6k+1$, so these numerical values satisfy $2g_{\mathrm{mirror}}=g_{\mathrm{cat}}+1$.  This is a family-specific numerical observation, not a consequence of Theorem~\ref{thm:intro-mirror}, which does not apply here.  No categorical or homological-mirror-symmetry explanation of the factor of two is established.
\end{remark}

\appendix
\section{The numerical inertial decomposition}\label{app:numerical-inertia}

Throughout, $X$ is a projective surface with isolated cyclic quotient singularities $\frac1{n_j}(1,q_j)$, and $\cX$ is its canonical stack.

By the convention in Definition~\ref{def:euler-pairing}, the Riemann--Roch theorem on $G_0$ applies to $K_0$.  No identification with vector-bundle $K$-theory is required.  Pullbacks, tensor products, and duals of coherent complexes below are derived.  Since $\cX$ is smooth, bounded coherent complexes are perfect, so the derived tensor product and derived dual used below remain in $\Db(\coh\cX)$.

The connected components of $I\cX$ are smooth in characteristic zero.  On a component $F$ of dimension $d$, cap product with $[F]$ identifies $A^p(F)$ with $A_{d-p}(F)$. We therefore write the Chow-homological Riemann--Roch maps in cohomological notation $A^*$.  Complexification separates the finite stabilizer representations into character eigenspaces and does not change the ranks used in Section~\ref{sec:serre}.

Working over $\C$ does not change the numerical quotient. More precisely, for $V=K_0(\cX)_{\Q}$, the radicals satisfy
\[
\operatorname{rad}(\chi_{\C})=(\operatorname{rad}\chi_{\Q})\otimes_{\Q}\C.
\]
Indeed, write a complexified class as a finite sum with $\Q$-linearly independent complex coefficients.  If its Euler pairing with every rational class vanishes, each rational class appearing in the sum belongs to the radical.  Thus extension of scalars commutes with passing to numerical equivalence.

The argument passes through the following maps, whose numerical compatibility and sector decomposition will be justified separately:
\[
\begin{aligned}
K_0(\cX)_{\C}&\xrightarrow{\ \widetilde{\operatorname{ch}}\ } A^*(I\cX)_{\C},\\
K_0^{\mathrm{num}}(\cX)_{\C}&\simeq A^*_{\mathrm{num}}(I\cX)_{\C},\\
A^*_{\mathrm{num}}(I\cX)_{\C}&\simeq A^*_{\mathrm{num}}(\cX)_{\C}
\oplus\bigoplus_j\bigoplus_{i=1}^{n_j-1}\C_{j,i}.
\end{aligned}
\]

\begin{definition}[Inertia morphism]
Recall that the inertia stack $I\cX$ parametrizes pairs $(x,g)$, where $g$ is an automorphism of $x$.  The identity automorphism gives the \emph{untwisted sector}, isomorphic to $\cX$. The connected components with $g\ne1$ are the \emph{twisted sectors}.  The inertia map $e:I\cX\to\cX$ forgets the automorphism.
\end{definition}

We use the convention of \eqref{eq:inertial-character-convention}.  The resulting inertial Chern character is multiplicative and extends linearly to $K_0(\cX)_{\C}$.

\begin{definition}[Numerical Chow groups]
For a smooth proper stack $Z$, possibly disconnected, write $A^*_{\mathrm{num}}(Z)_{\C}$ for its Chow group modulo the radical of the intersection pairing $\langle\alpha,\beta\rangle=\int_Z\alpha\beta$.  Integration sums the degrees of the zero-dimensional parts over the connected components.  In particular, $\alpha$ is numerically trivial precisely when this integral vanishes for every $\beta$.
\end{definition}

On the untwisted component we use the ordinary intersection product on the smooth stack $\cX$.  Lemma~\ref{lem:untwisted-rank} compares it with the Chow theory of the coarse moduli space $X$ via Vistoli's rational correspondence.

The first two lemmas record the standard Riemann--Roch machinery: the inertial Chern character is an isomorphism, and its expression for the Euler pairing identifies the numerical radicals.  The next two lemmas specialize to the cyclic quotient geometry, describing the inertia components and their numerical Chow groups.  A final lemma checks compatibility with tensoring by the canonical bundle, which is the feature needed for the Serre calculation in Section~\ref{sec:serre}.

\begin{lemma}[Inertial Chern character and Riemann--Roch]\label{lem:app-inertial-ch}
Let $N_e=(e^*T_{\cX})^{\mathrm{mov}}$ be the normal bundle of the inertia map, and let $t$ denote Edidin's eigenvalue-twisting operator from \cite[Section 4.3.2]{Edidin}, with the convention of \eqref{eq:inertial-character-convention}.  The inertial Riemann--Roch isomorphism $\tau={}^I\tau_{\cX}$ of \cite[Theorem 5.1]{Edidin} has the form
\[
\tau(v)=\widetilde{\operatorname{ch}}(v)\,\mathcal T,
\qquad
\mathcal T|_F=
\frac{\operatorname{td}(T_F)}
{\operatorname{ch}\bigl(t(\lambda_{-1}(N_e^\vee|_F))\bigr)}
\]
on each component $F\subset I\cX$, where $\lambda_{-1}(N_e^\vee)$ is the alternating sum $\sum_a(-1)^a[\bigwedge^aN_e^\vee]$.  The class $\mathcal T$ is a unit in $A^*(I\cX)_{\C}$.  Consequently, the inertial Chern character is an isomorphism of complex vector spaces:
\[
\widetilde{\operatorname{ch}}:K_0(\cX)_{\C}\xrightarrow{\sim}A^*(I\cX)_{\C}.
\]
\end{lemma}

\begin{proof}
The quotient-stack theorem applies: stabilizers act faithfully on $T_{\cX}$, so its frame bundle is an algebraic space presenting $\cX$ as a quotient by $\operatorname{GL}_2$.  Edidin \cite[Theorem 5.1]{Edidin} supplies the isomorphism on $G_0$ with formula
\[
\tau(v)=\operatorname{ch}\!\left(t\!\left(
\frac{e^*v}{\lambda_{-1}(N_e^\vee)}\right)\right)
\operatorname{td}(T_{I\cX}).
\]
Here $e$ and $N_e^\vee$ are his $f$ and $N_f^*$. Multiplicativity and $\widetilde{\operatorname{ch}}(v)=\operatorname{ch}(t(e^*v))$ give the stated factorization.  Krishna--Sreedhar \cite[Theorem 1.2]{KrishnaSreedhar} provide the more general Atiyah--Segal/Riemann--Roch framework.

On the untwisted component, $N_e=0$ and $\mathcal T=\operatorname{td}(T_{\cX})$ has constant term $1$.  On a twisted component, every moving eigenvalue differs from $1$, so the denominator has nonzero constant term. The Todd numerator has constant term $1$.  Thus $\mathcal T$ has nonzero degree-zero term on every component.  Positive-codimension classes are nilpotent, so $\mathcal T$ is a unit.  Multiplication by $\mathcal T^{-1}$ turns $\tau$ into the asserted Chern-character isomorphism.
\end{proof}

To pass to numerical equivalence, we must express dualization in these coordinates.

\begin{definition}[Duality on the inertia Chow group]
Let $\iota:I\cX\to I\cX$ send $(x,g)$ to $(x,g^{-1})$.  This is an involutive automorphism of the inertia stack, so $\iota^*$ is an isomorphism of its Chow groups.  For a class $\alpha=\sum_p\alpha_p$, where each $\alpha_p$ belongs to $A^p(I\cX)_{\C}$, define
\[
\alpha^\dagger=\iota^*\left(\sum_p(-1)^p\alpha_p\right).
\]
\end{definition}
Dualization reverses the character eigenvalues and multiplies the codimension-$p$ Chern character by $(-1)^p$.  Hence this $\C$-linear involution, which is not complex conjugation, encodes duality as
\[
\widetilde{\operatorname{ch}}(v^\vee)=\widetilde{\operatorname{ch}}(v)^\dagger.
\]

\begin{lemma}[Compatibility with numerical equivalence]\label{lem:app-numerical}
The inertial Chern character induces an isomorphism
\[
K_0^{\mathrm{num}}(\cX)_{\C}\simeq A^*_{\mathrm{num}}(I\cX)_{\C}.
\]
\end{lemma}

\begin{proof}
Let $p:\cX\to\operatorname{Spec}\C$ be the structure morphism and put
\[
a=v^\vee\otimes^{\mathbf L}w.
\]
By the proper-pushforward compatibility of inertial Riemann--Roch on $G_0$ for $p$ \cite[Theorem 5.1 and Corollary 5.3]{Edidin}, the Euler pairing satisfies
\[
\chi(v,w)=\int_{I\cX}{}^I\tau_{\cX}(a).
\]
Using Lemma~\ref{lem:app-inertial-ch}, multiplicativity of the inertial Chern character, and the duality convention above, this becomes
\begin{equation}\label{eq:inertial-euler-pairing}
\chi(v,w)=\int_{I\cX}
\widetilde{\operatorname{ch}}(v)^\dagger
\widetilde{\operatorname{ch}}(w)\,\mathcal T.
\end{equation}
Write $A=A^*(I\cX)_{\C}$ and let $N$ be the radical of its intersection pairing.  We check explicitly that the two operations modifying that pairing in \eqref{eq:inertial-euler-pairing} preserve $N$.

First, $N$ is an ideal: if $\alpha$ belongs to $N$ and $u,\beta$ belong to $A$, then the integral $\int\alpha u\beta$ vanishes, since $u\beta$ is another test class.  Multiplication by $\mathcal T$ therefore preserves $N$, and so does multiplication by $\mathcal T^{-1}$.  Thus $\alpha\mathcal T$ belongs to $N$ if and only if $\alpha$ belongs to $N$.

Second, $N$ is graded, because classes of each codimension can be tested against complementary codimensions on each component.  Changing the codimension-$p$ part by $(-1)^p$ therefore preserves $N$.  The isomorphism $\iota$ permutes the components and preserves intersection degrees, so $\iota^*$ also preserves $N$.  Since $\dagger$ is an involution, $\alpha^\dagger$ belongs to $N$ exactly when $\alpha$ belongs to $N$.

Finally, put $\alpha=\widetilde{\operatorname{ch}}(v)$.  By Lemma~\ref{lem:app-inertial-ch}, the classes $\widetilde{\operatorname{ch}}(w)$ run through all of $A$.  Multiplication by the unit $\mathcal T$ leaves this set of test classes unchanged.  Consequently, the following equivalences hold:
\[
\begin{aligned}
\chi(v,w)=0\text{ for every }w
&\quad\Longleftrightarrow\quad
\int_{I\cX}\alpha^\dagger\beta=0\text{ for every }\beta\in A\\
&\quad\Longleftrightarrow\quad \alpha^\dagger\in N
\quad\Longleftrightarrow\quad \alpha\in N.
\end{aligned}
\]
Thus Euler-numerical triviality is exactly numerical triviality under the inertial Chern character.  Quotienting the isomorphism of Lemma~\ref{lem:app-inertial-ch} by these corresponding radicals proves the result.
\end{proof}

The Riemann--Roch step is complete.  We now use the specific geometry of isolated cyclic quotient points to identify the numerical summands.

\begin{lemma}[Twisted sectors of isolated cyclic quotient points]\label{lem:app-sectors}
The inertia stack of $\cX$ decomposes as
\[
I\cX=\cX\sqcup\coprod_j\coprod_{i=1}^{n_j-1}\mathcal B\mu_{n_j}^{(i)}.
\]
\end{lemma}

\begin{proof}
The canonical stack has trivial stabilizers away from the singular points.  Near a point of type $\frac1n(1,q)$, a nonidentity element $\zeta^i\in\mu_n$ has tangent eigenvalues $\zeta^i$ and $\zeta^{iq}$.  Both differ from $1$ because $i$ lies between $1$ and $n-1$ and $n$ and $q$ are coprime, so its fixed locus is the origin.  Since $\mu_n$ is abelian, each nonidentity element gives one inertia component $[\{0\}/\mu_n]=\mathcal B\mu_n$.  Together with the identity sector, these are all the components.
\end{proof}

\begin{lemma}[Numerical Chow groups of twisted gerbes]\label{lem:app-gerbe-chow}
For a twisted component $\mathcal B\mu_n$, the numerical Chow group satisfies
\[
A^*_{\mathrm{num}}(\mathcal B\mu_n)_{\C}\simeq\C,
\]
and is concentrated in codimension zero.
\end{lemma}

\begin{proof}
Vistoli's rational comparison with the coarse moduli space \cite[Proposition 6.1]{Vistoli} identifies the Chow groups of $\mathcal B\mu_n$ with those of a point.  Thus $A^*(\mathcal B\mu_n)_{\C}$ is one-dimensional in codimension zero, with no positive-codimension part.  Its fundamental class has degree $\deg[\mathcal B\mu_n]=1/n\ne0$ \cite[Section 3.2.1]{Edidin}, so it survives numerical equivalence.
\end{proof}

It remains to check that the sector decomposition is preserved by the Serre operator.

\begin{lemma}[Compatibility with the Serre operator]\label{lem:app-serre}
Under the inertial Chern character, tensor product with a line bundle $L$ acts by multiplication by $\widetilde{\operatorname{ch}}(L)$.  In particular, the Serre operator $S_{\cX}$ preserves every summand of the numerical inertial decomposition.
\end{lemma}

\begin{proof}
By the ideal property of the numerical radical established in Lemma~\ref{lem:app-numerical}, multiplication by $\widetilde{\operatorname{ch}}(L)$ preserves that radical and therefore descends to $A^*_{\mathrm{num}}(I\cX)_{\C}$.  Multiplicativity gives
\[
\widetilde{\operatorname{ch}}(v\otimes L)
=\widetilde{\operatorname{ch}}(v)\,\widetilde{\operatorname{ch}}(L).
\]
Products on the Chow ring of the disjoint union $I\cX$ are componentwise, so every inertia-sector summand is preserved.  On a surface stack the Serre functor is tensor product with $\omega_{\cX}[2]$. The shift acts trivially on $K_0$.  Taking $L=\omega_{\cX}$ proves the assertion.  Thus the rank of $1-S_{\cX}$ can be computed separately on each summand, as used in Section~\ref{sec:serre}.
\end{proof}

\begin{proof}[Proof of Proposition~\ref{prop:numerical-inertia}]
Lemma~\ref{lem:app-numerical} gives $K_0^{\mathrm{num}}(\cX)_{\C}\simeq A^*_{\mathrm{num}}(I\cX)_{\C}$.  Lemma~\ref{lem:app-sectors} decomposes the inertia into $\cX$ and the components $\mathcal B\mu_{n_j}^{(i)}$, whose numerical Chow groups are each $\C$ by Lemma~\ref{lem:app-gerbe-chow}.  Since the intersection pairing is componentwise, these identifications yield
\[
\boxed{
K_0^{\mathrm{num}}(\cX)_{\C}
\overset{\widetilde{\operatorname{ch}}}{\simeq}
A^*_{\mathrm{num}}(I\cX)_{\C}
\simeq A^*_{\mathrm{num}}(\cX)_{\C}
\oplus\bigoplus_j\bigoplus_{i=1}^{n_j-1}\C_{j,i}.}
\]
Lemma~\ref{lem:app-serre} makes this decomposition $S_{\cX}$-invariant.  This proves the intrinsic assertion about the stack. Vistoli's comparison with $X$ is used only afterwards, in the untwisted calculation of Lemma~\ref{lem:untwisted-rank}.
\end{proof}

\section*{Acknowledgements}

The author thanks Sergey Galkin for pointing out the genus-two phenomenon for the mirror of $X_{10}$ and suggesting its relation to the antisymmetric Euler pairing.  The author also thanks Sergei Burkin for teaching him the modern viewpoint on characteristic classes, which influenced the Chern-character and Riemann--Roch perspective used here.

\section*{Funding}

This research was supported by FAPERJ (Fundação Carlos Chagas Filho de Amparo à Pesquisa do Estado do Rio de Janeiro) through the Bolsa Nota 10 doctoral scholarship program, Chamada FAPERJ Nº 01/2025, Processo SEI-260003/001060/2025.

\section*{Declaration of generative AI and AI-assisted technologies in the writing process}

During the preparation of this manuscript, the author used ChatGPT (OpenAI) for language editing, organization, \LaTeX{} formatting, bibliographic research, mathematical exposition, and exploratory symbolic and numerical checks.  The author independently checked all mathematical statements, proofs, computations, and uses of references, and takes full responsibility for the final manuscript.

\printbibliography

\end{document}